\documentclass[a4paper,10pt]{amsart}
\usepackage[utf8]{inputenc}
\usepackage{amsthm}
\usepackage{comment}
\usepackage{amsmath}
\usepackage{bm}
\usepackage{amsfonts}
\usepackage{amssymb}
\usepackage{mathtools}
\usepackage{appendix}
\usepackage{mathrsfs}
\usepackage{setspace}
\usepackage{thmtools}
\usepackage[foot]{amsaddr}
\usepackage[obeyFinal,textwidth=2.7cm]{todonotes}
\usepackage[pdfdisplaydoctitle,colorlinks,breaklinks,urlcolor=blue,linkcolor=blue,citecolor=blue]{hyperref} 
\usepackage[nameinlink,capitalise]{cleveref}

\newcommand{\D}{\mathcal{D}}

\newcommand{\E}{\mathcal{E}}
\newcommand{\EE}{\mathbb{E}}
\newcommand{\F}{\mathcal{F}}

\renewcommand{\H}{\mathcal{H}}
\newcommand{\HH}{\mathbb{H}}

\newcommand{\LL}{\mathcal{L}}

\newcommand{\N}{\mathbb{N}}

\newcommand{\PP}{\mathbb{P}}

\newcommand{\R}{\mathbb{R}}

\newcommand{\XX}{\mathcal{X}}

\newcommand{\x}{\mathbf{x}}

\newcommand{\Do}{\mathsf{D}}

\newcommand{\xxi}{\bm \xi}

\DeclareMathOperator{\diam}{diam}

\newcommand{\zbgn}{Z_{\beta}^N}
\newcommand{\free}{{\textrm{\Tiny free}}}
\newcommand{\ebg}{e^{\frac{\beta}{2}V_m^\free(0)}}

\newcommand{\sumast}{\sideset{}{^\ast}\sum}

\renewcommand{\setminus}{\smallsetminus}

\newcommand{\one}{\bm{1}}

\newcommand{\de}{\partial}

\newcommand{\set}[1]{\left\{#1\right\}}
\newcommand{\pa}[1]{\left(#1\right)}

\newcommand{\braclose}[1]{\left.#1\right]}

\newcommand{\abs}[1]{\left|#1\right|}
\newcommand{\norm}[1]{\left\|#1\right\|}
\newcommand{\brak}[1]{\left\langle#1\right\rangle}

\newcommand{\expt}[2][]{\mathbb{E}_{#1}\left[#2\right]}

\newcommand{\texpt}[2][]{\tilde{\mathbb{E}}_{#1}\left[#2\right]}
\newcommand{\texptopen}[2][]{\tilde{\mathbb{E}}_{#1}\left[#2\right.}

\newtheorem{theorem}{Theorem}[section]
\newtheorem{definition}[theorem]{Definition}

\newtheorem{corollary}[theorem]{Corollary}
\newtheorem{lemma}[theorem]{Lemma}
\newtheorem{proposition}[theorem]{Proposition}

\theoremstyle{remark}
\newtheorem{remark}[theorem]{Remark}

\numberwithin{equation}{section}

\newenvironment{acknowledgements}{%
  \begin{abstract}
}{%
  \end{abstract}
}

\title[Fluctuations of Point Vortex Dynamics]{Gibbs Equilibrium Fluctuations\\ of Point Vortex Dynamics}
\author[F. Grotto]{Francesco Grotto}
\address{Università di Pisa, Dipartimento di Matematica, 5 Largo Bruno Pontecorvo, 56127 Pisa, Italia.}
\email{\href{mailto:francesco.grotto at unipi.it}{francesco.grotto at unipi.it}}
\author[E. Luongo]{Eliseo Luongo}
\address{Scuola Normale Superiore, Piazza dei Cavalieri, 7, 56126 Pisa, Italia}
\email{\href{mailto:eliseo.luongo at sns.it}{eliseo.luongo at sns.it}}
\author[M. Romito]{Marco Romito}
\address{Università di Pisa, Dipartimento di Matematica, 5 Largo Bruno Pontecorvo, 56127 Pisa, Italia.}
\email{\href{mailto:marco.romito at unipi.it}{marco.romito at unipi.it}}
\date{July 31, 2023}

\begin{document}

\begin{abstract}
    We consider a system of $N$ point vortices in a bounded domain with null total circulation, whose statistics are given by the Canonical Gibbs Ensemble at inverse temperature $\beta\geq 0$. We prove that the space-time fluctuation field around the (constant) Mean Field limit satisfies when $N\to\infty$ a generalized version of 2-dimensional Euler dynamics preserving the Gaussian Energy-Enstrophy ensemble.
\end{abstract}

\maketitle

\section{Introduction}\label{sec:introduction}

Existence and uniqueness of solutions for 2D Euler equations,
\begin{equation*}
    \begin{cases}
        \de_t u + (u\cdot\nabla)u + \nabla p = 0,\\
    \nabla\cdot u       = 0,
    \end{cases}
\end{equation*}
is, unlike the 3D case, a settled problem at least in the Judovic well-posedness class \cite{Judovic1963}. For more general initial conditions non-uniqueness phenomena are known to occur \cite{Vis2018a,Vis2018b,BruCol2021}. Here we are interested in the existence of 2D Euler dynamics with a special class of initial distributions of much rougher initial conditions given by the so-called \emph{Energy-Enstrophy measures}, formally defined by
\begin{equation}\label{eq:energyenstrophymeasure}
  d\mu_\beta(\omega)
    = \frac1{Z_\beta}\exp\pa{-\beta \int |u|^2\,dx-\int \omega^2\,dx}d\omega, 
\end{equation}
where $u$ is the fluid velocity, $\omega=\nabla^\perp\cdot u$ the associated vorticity, and $Z_\beta$ is a normalising constant. These measures are of paramount physical importance as (kinetic) energy and enstrophy are the main conserved quantities and thus such measures are left invariant, at least formally, by the Euler dynamics.

Our starting point is the proof of existence of solutions preserving Energy-Enstrophy measures of \cite{AlCr90}, which relied on finite dimensional Galerkin approximations while leaving unanswered the question of the physical validity of these solutions, that is whether they can be obtained by approximation with smooth solutions. This problem was later addressed in \cite{Fla2018}, under periodic boundary conditions and in the restricted case of white noise initial condition corresponding to infinite temperature $\beta=0$, the \emph{Enstrophy measure}. \cite{Fla2018} proved that weak solutions of 2D Euler equations preserving the enstrophy measure can be obtained as limit of regular solutions by first approximating the former with a point vortex system, in which initial vortex positions were \emph{independently} distributed, thus recovering the Gaussian enstrophy measure at fixed time in a classical central limit theorem. 

The main aim of this paper is to prove that the Energy-Enstrophy measures can be obtained for all values of $\beta>0$ in \eqref{eq:energyenstrophymeasure} as limit of point vortex dynamics preserving Gibbs statistical ensembles, where vortex positions are \emph{not independent}, thus providing a complete and rigorous framework for the problem outlined above. Moreover, we consider the more challenging problem of the motion on a bounded domain, where boundary effects must be taken into account in the point vortex dynamics. Our proof is completely probabilistic in nature and is based on the interpretation of solutions as space-time fluctuations around a (null) mean field limit of interacting vortices.
\bigskip

Point Vortex (PV) systems are classical models dating back to the work of Helmholtz \cite{Helmholtz1858} and Kirchhoff \cite[Lecture 20]{Kirchhoff1876}, whose importance is paramount both in the context of numerical approximation of 2-dimensional incompressible flows and in their theoretical description \cite{MaPu94}. The study of PV as a (singular) Hamiltonian system, in particular concerning integrability properties or singular solutions of the system, also constitutes a well-established research subject \cite{Aref2007}.

We consider the PV system on a bounded domain $D\subset \R^2$ with smooth boundary. Let $G$ be the Green function of the Laplace operator $-\Delta$ with Dirichlet boundary conditions on $D$, and denote by $K=\nabla_x^\perp G(x,y)$ the Biot-Savart kernel. If $x_i$ are the vortex \emph{positions} and $\xi_i$ their \emph{intensities}, the equations of motions are
\begin{equation}\label{eq:pv}\tag{PV}
    \dot x_i(t)=\frac{1}{\sqrt{N}}\sum_{j\neq i}^N\xi_j K(x_i(t),x_j(t))+\frac{1}{\sqrt{N}}\xi_i\nabla^{\perp}g(x_i(t),x_i(t)),
    \quad i=1,\dots,N,
\end{equation}
where $g$ takes into account the effect of the boundary, \emph{cf.} \eqref{harmonic_extension_green} below, and prefactors $1/\sqrt N$ fix intensities to the fluctuation scale. A series of works initiated with \cite{Marchioro1983,Marchioro1988,Marchioro1993,Marchioro1994} (see \cite{Butta2018} for a later contribution) provided a derivation of \eqref{eq:pv} as a limit of \emph{vortex patch} solutions of 2-dimensional Euler equations in the vorticity formulation,
\begin{equation}\label{eq:2deuler}\tag{2dE}
    \de_t \omega + (K\ast \omega)\cdot \nabla\omega=0,
\end{equation}
and convergence of \eqref{eq:pv} to solutions of \eqref{eq:2deuler} with $L^\infty$ vorticity (the well-posedness class of Judovic, \cite{Judovic1963}) is the object of recent relevant results \cite{Rosenzweig2022}.

The statistical mechanics approach to point vortices dates back to Onsager \cite{onsager}, motivated by the description of turbulent phenomena. The influential works \cite{Caglioti1992,Caglioti1995,Lions1998} rigorously established the convergence of statistical ensembles of PV in the Mean Field Limit to stationary solutions of \eqref{eq:2deuler}, characterized by the so-called \emph{Mean Field equation}.
In this paper we restrict our discussion to the case in which vortex intensities have the same magnitude $\xi_i=\pm1$ and are subject to the \emph{neutrality condition}
\begin{equation}\label{eq:neutrality}\tag{NC}
    \sum_{j=1}^N \xi_j = 0.
\end{equation}
Under this assumption one-particle marginals of the (Mean-Field rescaled) Canonical Gibbs ensemble,
\begin{equation}\label{eq:gibbs}
    d\nu^N_\beta (x_1,\dots, x_N)
      = \frac1{Z_\beta^N}  \exp\pa{-\frac\beta{N} H(x_1,\dots, x_N)}\,dx_1\cdots\,dx_N,
\end{equation}
with Hamiltonian function
\begin{equation}\label{eq:hamiltonian}
    H(x_1,\dots,x_n)
      =\sum_{j=1}^N\sum_{i< j}^N \xi_i\xi_j G(x_i,x_j)+\frac{1}{2}\sum_{i=1}^{N}\xi_i^2 g(x_i,x_i),
\end{equation}
converge to the uniform distribution on $D$ \cite{Caglioti1992,Caglioti1995}, while particle correlations decay to zero (\emph{cf.} \cite{Grotto2020decay} for a quantitative result). Therefore, in this scaling regime, the empirical measure $\frac1N\sum_{i=1}^N\xi_i \delta_{x_i}$ converges to $0$ under the Canonical Ensemble.

Building up on the study initiated in \cite{GrRo19}, we consider the \emph{fluctuation field}
\begin{equation}\label{eq:defvort}
    \omega^N_t=\frac1{\sqrt N}\sum_{i=1}^N \xi_i \delta_{x_i},
\end{equation}
describing fluctuations at equilibrium around the null Mean-Field limit of the empirical measure.
Completing previous results \cite{BePiPu87,Bodineau1999}, \cite{GrRo19} established the convergence \emph{at fixed time} of $\omega^N_t$ under the Canonical Ensemble to the Energy-Enstrophy measure.

In this paper we establish convergence of the full space-time fluctuation field $(\omega^N_t)_{t\in [0,T]}$ of PV dynamics preserving $\nu^N_\beta$ at finite temperature $\beta> 0$ to weak vorticity solutions of Euler equations preserving the Energy-Enstrophy measure. Our main results, the detailed statement of which we defer to the forthcoming \cref{sec:preliminaries}, can be summarized as follows.
\begin{enumerate}
  \item In \cref{thm:spacebounds,thm:timebounds} we prove uniform bounds, in $N\to\infty$, for nonlinear functionals of the fluctuation field $\omega^N_t$ when positions $x_i$ are distributed according to $\nu^N_\beta$. In particular, we provide uniform $L^p$ estimates of functionals of the form
\begin{equation*}
    \sum_{i,j}^N\xi_i \xi_j f(x_i,x_j).
\end{equation*}
These local bilinears of the fluctuation field play a crucial role: a proper choice of $f$ allows to write in that form   (negative order) Sobolev norms of the fluctuation field, and, perhaps most importantly, sums of two-particle interactions which we need to control in order to study the limiting dynamics as $N\to \infty$. \cref{thm:spacebounds,thm:timebounds} develop the arguments of \cite{GrRo19} and improve results in the latter, the proof relying on a transformation converting integrals over phase space to expectations of functionals of a Gaussian random field.
  \item \cref{thm:limit} establishes the convergence of the flow $t\mapsto \omega^N_t$ to a stochastic process $t\mapsto \omega_t$ whose trajectories satisfy \eqref{eq:2deuler} in a weak sense to be described in \cref{subsec:weak_vorticity_formulation}. No external forcing (\emph{i.e.} noise) is involved in the dynamics, and stochasticity is only due to the invariant measure from which the initial datum is sampled. Indeed, according to \cref{thm:limit}, \emph{the limiting dynamics can be described as a measure-preserving weak solution of} \eqref{eq:2deuler}. We refrain from using the term \emph{flow} in this context, since no uniqueness result (or in fact even a proper notion of it) is available at this regularity regime, it actually being a long-standing open problem \cite{Albeverio2008}. 
\end{enumerate}
The limit fluctuation dynamics $t\mapsto \omega_t$ possesses some interesting and distinguished characteristics. \emph{The evolution is still non-linear in the limit}, and it satisfies a generalized version of the equations of which PV are a discretization. As a consequence, while single-time marginals of $\omega_t$ have Gaussian distribution corresponding to the Energy-Enstrophy measure of \cref{eq:2deuler}, because of the CLT of \cite{GrRo19}, \emph{multi-time distributions are not jointly Gaussian};we refer to \cref{lem:nongaussian} below for a precise statement. 

\section{Background Material and Main Results}\label{sec:preliminaries}

We assume from now on that $D\subset \R^2$ is a bounded convex domain with smooth boundary.
We also assume that $D$ has Lebesgue measure $\int_{D}dx=1$;
this is without loss of generality in all the forthcoming arguments.

\subsection{Notation}
We will often write for brevity
\begin{gather*}
    \x_N=(x_1,\dots,x_N)\in D^N,\quad 
    \xxi_N=(\xi_1,\dots,\xi_N)\in \set{\pm 1}^N\\
    \int_{D^N}f(\x_N)d\x_N=\int_{D^N}f(x_1,\dots,x_N)dx_1\cdots dx_N.
\end{gather*}
We denote $\D=C^\infty_c(D)$ and by $\D'=C^\infty_c(D)'$ its topological dual, \emph{i. e.} the space of distributions. Brackets $\brak{\cdot,\cdot}$ will denote duality couplings
in $\D\times \D'$, to be understood as $L^2$ products whenever both arguments belong to the latter.  

\subsection{Function Spaces over Bounded Domains}\label{sec:sobolevspace}

We will make extensive use of fractional Sobolev spaces and fractional powers
of the Dirichlet Laplacian. The latter can be defined in various ways: 
we choose the so-called \emph{spectral fractional Laplacian}.
We refer to \cite{Caffarelli2016}, \cite[Chapter 3]{Triebel83} for a detailed treatment of the topic discussed in this Section.

By $H^s_0(D)=W^{s,2}_0(D)$ (resp.  $H^s(D)=W^{s,2}(D)$), $s\in \N$, we denote Sobolev spaces
obtained as the closure of $C^\infty_c(D)$ (resp. $C^s(\bar{D})$) under the usual $L^2$-based Sobolev norm. For $s=0$, we set $H^0(D)=H_0^0(D)=L^2(D)$. Recall now the usual definition of fractional Sobolev spaces $H^s_0(D)$ (resp $H^s(D)$), $0<s<1$, as the completion
of $C^\infty_c(D)$ (respectively $C^{1}(\bar{D})$) under the Sobolev-Slobodeckij norm
\begin{equation*}
    \norm{u}^2_{L^2(D)}+\int_D \int_D \frac{(u(x)-u(y))^2}{|x-y|^{2+2s}}dxdy.
\end{equation*}
For $0\leq s\leq 1/2$ it can be shown that actually $H^s_0(D)=H^s(D)$, while for $1/2<s\leq $ $H^s_0(D)\hookrightarrow H^s(D)$ the trace operator is well defined and the fractional Sobolev spaces $H^s_0(D)$ are done by elements having null trace.
The space $H_0^{1+s}(D)$ (resp. $H^{1+s}(D)$), $s\in (0,1)$ is the Hilbert space of functions 
$u\in H^1_0(D)$ (resp. $u\in H^1(D)$) for which $u,\nabla u\in H_0^s(D)$ (resp. $u,\nabla u\in H^s(D)$).
Iterating this procedure, we can define $H^s_0(D)$ (resp. $H^s(D)$) for each $s\geq 0$. 
For $s>0$ the topological dual of $H^{s}_0(D)$,  $H^{-s}(D)$, is a subspace of $\mathcal{D}'$, since the latter is obtained as the closure of $C^{\infty}_c(D)$ with respect to a suitable norm. The following collects classical results, specialized to the two dimensional case, on the relations between the spaces defined above.
\begin{lemma} 
    The embeddings $H^s_0(D)\hookrightarrow L^2(D)\hookrightarrow H^{-s}(D)$ and $H^\alpha(D)\hookrightarrow H^\beta(D)$, $\alpha>\beta$, are continuous, dense and compact. 
    Moreover, it holds:
	\begin{itemize}
        \item (Interpolation, \cite[Chapter 1]{LioMag}) if $\theta\in [0,1]$ and $(1-\theta)s+\theta s'\notin \N+1/2$ we have \begin{align*}
    \left[H^s_0(D), H^{s'}_0(D)\right]_{\theta}&=H^{(1-\theta)s+\theta s'}_0(D),\\ \left[H^{-s}(D), H^{-s'}(D)\right]_{\theta}&=H^{-(1-\theta)s+\theta s'}(D) ;
    \end{align*}
    \item  (Hilbert-Schmidt's embedding, \cite{Maurin61,Cla66}) If $t>1$ and $m\in\N$ the embedding 
    \begin{align*}
        H^{m+t}_0(D)\hookrightarrow H^m_0(D)
    \end{align*}
    is of Hilbert-Schmidt type.
	\end{itemize}
\end{lemma}
In what follows we will use also a different definition of fractional Sobolev spaces
based on Fourier series expansions.
There exists an orthonormal basis of $L^2(D)$ of eigenfunctions $\set{e_n}_{n\in\N}\subset H^1_0(D)$
of the Dirichlet Laplacian $-\Delta$, $-\Delta e_n=\lambda_n e_n$.

For $s\geq 0$, spectral fractional Laplacian is defined by
\begin{gather}\nonumber
	\mathsf{D}\pa{(-\Delta)^s}=\set{u=\sum_{n=1}^\infty \hat u_n e_n, \hat u_n\in \R: 
		\sum_{n=1}^\infty \lambda_n^{2s} |\hat u_n|^2<\infty}\\
	\label{eq:spectralfractional}
	(-\Delta)^s u= \sum_{n=1}^\infty \lambda_n^{s} \hat u_n e_n\in L^2(D),\quad \hat{u}_n=\brak{ u,e_n} ,\quad u\in \Do\pa{(-\Delta)^s},
\end{gather}
as a densely defined, positive self-adjoint operator on $L^2(D)$. For $s>0$, we define $\H^s=\mathsf{D}((-\Delta)^{s})$,
which is a Hilbert space with norm
\begin{equation}\label{norm H^s}
    \norm{u}_{\H^s}^2=\brak{(-\Delta)^{s/2}u,(-\Delta)^{s/2} u}.
\end{equation}
It is proved in \cite{Caffarelli2016} that $\H^s=H^s(D)$ for $s\in (0,1/2)$,
$\H^s=H^s_0(D)$ for $s\in (1/2,3/2)$, in the sense that they coincide as function spaces 
and their norms are equivalent.
For $s=1/2$ the space $\H^{1/2}$ coincides instead with the so-called Lions-Magenes space $H^{1/2}_{00}(D)$,
a fact we can neglect in our discussion.

Negative order fractional Laplacian $(-\Delta)^s$ and spaces $\H^s$, $s<0$, can be defined just as in \eqref{eq:spectralfractional},
with $\H^s=\Do\pa{(-\Delta)^s}$, $s<0$. This is equivalent to say that for $s<0$, $\H^s$ is the closure of $L^2(D)$ with respect to the norm \eqref{norm H^s}. In particular, for $0<s<3/2,\ s\neq 1/2$, $\H^{-s}=H^{-s}$ with equivalent norms. Indeed, as discussed above, since $H^{s}_0(D)\hookrightarrow L^2(D)$ with dense embedding, then also $L^2(D)\hookrightarrow \left(H^{s}_0(D)\right)^*=H^{-s}$ with dense embedding. Therefore we are left to show that \begin{align*}
        \norm{ u}_{\mathcal{H}^{-s}}\simeq\norm{ u}_{H^{-s}}\quad \forall u\in L^2(D).
    \end{align*}
    The latter is true indeed it holds
    \begin{align*}
       \norm{ u}_{\mathcal{H}^{s}}=\operatorname{sup}_{f\in L^2(D)}\frac{\brak{ (-\Delta)^{-s}u,f}}{\norm{ f}_{L^2}}\stackrel {g=(-\Delta)^{-s}f}{\simeq} \operatorname{sup}_{g\in H^{s}_0}\frac{\brak{ u,g}}{\norm{ g}_{H^{s}_0}} = \norm{ u}_{H^{s}_0}.
    \end{align*}
Let us introduce some further function spaces:
\begin{align*}
H^{-1-}(D)=\cap_{\delta>0}H^{-1-\delta}(D),\quad \XX=C([0,T];H^{-1-}).
\end{align*}
Both $H^{-1-}(D)$ and $\XX$ are Polish spaces metrized by the distances
\begin{align*}
    d_{H^{-1-}}(\omega,\omega')&=\sum_{n=1}^{+\infty}\frac{1}{2^n}\pa{\norm{ \omega-\omega'}_{H^{-1-1/n}}\wedge 1},\\
    d_{\XX}(\omega_\cdot,\omega_\cdot')&=\sup_{t\in [0,T]}d_{H^{-1-}}(\omega_t,\omega'_t).
\end{align*}
Convergence of $\{\omega^N\}_{N\in\N}$ to $\omega$ in $H^{-1-}(D)$ (resp. $\XX$) with respect to the distance $d_{H^{-1-}}(\cdot,\cdot)$ (resp. $d_{\XX}(\cdot,\cdot)$) is equivalent to that of $\{\omega^N\}_{n\in\N}$ to $\omega$ in $H^{-1-\delta}(D)$ (resp. $C([0,T];H^{-1-\delta}(D))$) for each $\delta>0$.\\
Lastly we introduce 
\begin{align*}
   L^{\infty-}(0,T;H^{-1-}(D))=\cap_{\delta>0,\ q\in (1,+\infty)}L^q(0,T;H^{-1-\delta}(D)),
\end{align*}
which is a Polish space, with distance
\begin{align*}
    d_{L^{\infty-}(H^{-1-})}(\omega_\cdot,\omega_\cdot')=\sum_{n=1}^{+\infty}\frac{1}{2^n}\left(\left(\int_0^T \norm{ \omega_t-\omega'_t}_{H^{-1-\frac{1}{n}}}^n dt\right)^{1/n}\wedge 1\right).
\end{align*}

\subsection{Green Functions}
As an integral kernel, the Green function of the spectral fractional Dirichlet Laplacian $(-\Delta)^s$ takes the form
\begin{equation*}
G_s(x,y)=\sum_{n\in\N} \lambda_n^{-s} e_n(x)e_n(y),\quad s\in\R.
\end{equation*}
For $s=1$, $G=G_1=(-\Delta)^{-1}$ has a logarithmic singularity, it can be represented as
\begin{equation}\label{greendomain}
G(x,y)=-\frac{1}{2\pi}\log|x-y|+g(x,y)  \quad y\in D, 
\end{equation}
being $g$ the harmonic continuation in $D$ of $-\frac1{2\pi}\log$ on $\partial D$, i.e. the solution of 
\begin{equation}\label{harmonic_extension_green}
\begin{cases}
\Delta g(x,y)=0 & x\in D\\
g(x,y)=\frac{1}{2\pi}\log|x-y| &x\in \partial D.
\end{cases}
\end{equation}
Both $G$ and $g$ are symmetric, and maximum principle implies that
\begin{equation}\label{harmoniccontbound}
\frac{1}{2\pi}\log(d(x)\vee d(y))\leq g(x,y)\leq \frac{1}{2\pi} \log\diam(D),
\end{equation}
with $d(x)$ the distance of $x\in D$ from the boundary $\partial D$.

For $s>1$, the kernel $G_s(x,y)$ defined above is a bounded function on $\bar D$.
Estimates on the singularity of $G_s$ at $x=y$ in the case $s\in(-1,0)$ (which we will not need)
can be found in \cite[Section 2.5]{Caffarelli2016}.

\subsection{Point Vortex Dynamics and Canonical Ensembles}\label{subseq:PVdynamics}
Well-posedness of the singular ODE system \eqref{eq:pv} for almost every initial condition is a classical result, first established on the torus in \cite{DuPu82} and on the full plane in \cite{MaPu94}: the forthcoming statement is a simple generalization of those results, as mentioned in \cite[Section 2]{DuPu82}. We also refer to \cite{Donati2022} for the result on more general space domains, that includes the case we are considering.

\begin{proposition}\label{Well-posedness point vortices}
    Let $\xi_1,\dots,\xi_N\in\R^\ast$ and $N\geq 2$. For Lebesgue-almost all initial configurations $x_1(0),\dots,x_N(0)\in D$ of distinct points there exists a unique, smooth solution $\R\ni t\to(x_1(t),\dots,x_N(t))\in D^N$ to \eqref{eq:pv} such that $x_i(t)\neq x_j(t)$ for all $t$ and distinct $i,j$, that is trajectories never intersect nor touch the boundary of $D$. The corresponding (almost-everywhere defined) solution flow preserves Lebesgue measure on $D^N$.      
\end{proposition}

It is worth remarking that notwithstanding the latter result, singular trajectories in which multiple vortices collapse or burst out of a single one \cite{GrottoPappalettera2022}; we refer to \cite{GrottoJFA2020} for further considerations on the issue of uniqueness in PV dynamics.

Since \eqref{eq:pv} is a Hamiltonian system in conjugate coordinates $(x_{i,1},\xi_i x_{i,2})_{i=1,\dots ,N}$ with Hamiltonian function $H$ defined in \eqref{eq:hamiltonian}, the Canonical Gibbs measure $\nu^N_\beta$ is preserved by the flow of \eqref{eq:pv} provided that it is well defined.
This is always true if $N$ is large enough, as it holds:

\begin{proposition}\label{prop:zgreendomain} \cite[Proposition 4.1]{GrRo19}
	If
	\begin{equation*}
		-8\pi\frac{N}{\max(n_+,n_-)} < \beta < 2\pi N,
	\end{equation*}
	then $Z_{\beta,N}<\infty$, and the measure $\nu_\beta^N$ 
	is thus well-defined.
\end{proposition}

The technical core of this paper consists in uniform bounds in suitably strong functional norms of the fluctuation field $\omega^N$, in the form of the following Theorems.

\begin{theorem}\label{thm:spacebounds}
    	Let $\beta\geq 0$ and $\delta>0$. For any $p\in \N$ it holds
	\begin{multline}\label{eq:lpbounds}
		\int_{D^N} \norm{\frac1{\sqrt N}\sum_{i=1}^N\xi_i \delta_{x_i}}^p_{H^{-1-\delta}} d\nu_\beta^N\\
        =\int_{D^N} \frac1{N^p}\pa{\sum_{i,j=1}^N \xi_i\xi_j G_{1+\delta}(x_i,x_j)}^p d\nu_\beta^N
        \leq C_{\beta,\delta,p}.
	\end{multline}
\end{theorem}

\begin{theorem}\label{thm:timebounds}
     For all $\beta \geq 0$ and any bounded measurable $f:D\times D\to \R$ symmetric in its argument and with null diagonal average,
    \begin{equation}\label{eq:fnulldiagonal}
    \int_D f(x,x)dx = 0,\\
  \end{equation}
    for all $\epsilon<\frac12$ it holds
    \begin{multline}\label{eq:corrbounds}
    \int_{D^N} \pa{\frac1N\sum_{i,j=1}^N \xi_i\xi_j f(x_i,x_j)}^2 d\nu^N_\beta(x_1,\dots, x_N)\\
      \lesssim_{\beta,\epsilon,D} \norm{ f}_{L^{2/\epsilon}}^2
      +\frac1N{\sup_{x\in \bar{D}}\abs{ f(x,x)}^2} +\frac1{N^{1-\epsilon}}{\sup_{x\in\bar{D}}\abs{ f(x,x)}\norm{ f}_{L^{1/\epsilon}}}.
    \end{multline}
\end{theorem}

The proof of these Theorems is the content of \Cref{sec:fluctuationbounds}. 
As we shall discuss, \eqref{eq:lpbounds} allows to show functional convergence to the Energy-Enstrophy Gaussian distribution at any fixed time, and its proof relies on a combinatorial argument that reduces the argument to a uniform bound on the partition function $\zbgn$. 

On the other hand, \eqref{eq:corrbounds} allows to consider the limit as $N\to\infty$ of the dynamics, by properly choosing the two-particle test function $h$ so to model vortex interactions that determine the time evolution. Its proof is considerably longer, since it is not possible to reduce it to the uniform bounds proved in \cite{GrRo19}. Our arguments will rely on techniques from Statistical Mechanics, and they can be summarized as follows:
\begin{itemize}
    \item a \emph{potential splitting} is used to separate the singular (albeit negligible in the Mean Field Limit) short-range contribution of the logarithmic potential, to be controlled by means of a careful analysis of phase space in the fashion of \cite{Deutsch1974};
    \item the partition function associated to the long-range, regular part of the interaction is transformed into a Gaussian integral, in what is commonly known as \emph{Sine-Gordon transformation} due to the fact that it establishes a correspondence between 2d Coulomb Gas statistical mechanics (equivalent to the one of point vortices) and Sine-Gordon quantum field theory, \cite{Frolich1976,Stu1978};
    \item the asymptotic expansion of said Gaussian integrals is performed by a \emph{Mayer cluster expansion} (\emph{cf.} in particular \cite{Brydges1978,Brydges1980} for the application of that technique in combination with Sine-Gordon transformation) reducing computations to estimates on correlation functions of finitely many Gaussian fields.
\end{itemize}

\subsection{Gaussian Analysis on Bounded Domains}
The 2-dimensional Euler equations \eqref{eq:2deuler} on $D$ have two quadratic first integrals, respectively \emph{energy} and \emph{enstrophy},
\begin{equation*}
    \frac12\int_D |u|^2dx=\frac12 \int_D \omega (-\Delta)^{-1}\omega dx,\qquad
    \frac12 \int_D \omega^2 dx.
\end{equation*}
The Energy-Enstrophy measure is formally defined by
\begin{equation}\label{eq:energyenstrophyfield}
    d\mu_\beta(\omega)= \frac1{Z_\beta}\exp\pa{-\beta \int_D |u|^2dx-\int_D \omega^2 dx}d\omega, 
\end{equation}
where, again, $u$ stands for the fluid velocity, that is $\nabla\cdot u=0$ and $\nabla^\perp\cdot u=\omega$ or, by Biot-Savart law, $u=K\ast \omega$.
The field-theoretical notation of above indicates the Gaussian measure induced by a linear combination of the two quadratic functionals associated to energy and enstrophy, which we now rigorously define.

In order to define \eqref{eq:energyenstrophyfield} as a centered Gaussian field we only need to specify a covariance operator. Such Gaussian field should also have zero space average, as in our context it appears as the limit of measure-valued objects (empirical measures of vortices) subjected to null space average condition \eqref{eq:neutrality}. Define the bounded linear operator
\begin{equation}\label{eq:meanoperator}
M:L^2(D)\rightarrow L^2(D), \quad Mf(x)=f(x)-\int_{D} f(y)dy.
\end{equation}
For $\beta\geq 0$, we denote by $\mu_\beta$ the law of the centered Gaussian random field on $D$ with covariance
\begin{equation*}
\forall f,g\in L^2(D), \quad 	
\expt[\mu_\beta]{\omega(f) \omega(g)}=\brak{f,Q_\beta g}_{L^2(D)}, 
\quad Q_\beta=M^*(I-\beta\Delta)^{-1}M.
\end{equation*}
Equivalently, under $\mu_\beta$, $\omega$ is a centered Gaussian stochastic process indexed by $L^2(D)$ with the specified covariance, and it can be identified with a random distribution taking values in $H^s(D)$ for all $s<-1$ (\emph{cf.} \cite[Section 4]{GrRo19}). The evaluation of such a random field on a function $f\in L^2(D)$ is therefore to be regarded as a (first order) stochastic integral, $\omega(f)=I^1_\omega(f)$.

Single and double stochastic integrals with respect to the Gaussian random measure $\mu_\beta$ (in the sense of \cite[Section 7.2]{Ja97}),
which we denote by
\begin{equation*}
    L^2(D)\ni f\mapsto I^1_{\omega}(f)\in L^2(\mu_\beta),\quad 
    L^2(D^2)\ni h\mapsto I^2_{\omega}(h)\in L^2(\mu_\beta),
\end{equation*}
provide isometries of Hilbert spaces, as it holds
\begin{equation}\label{eq:ItoIsometry}
    \int I^1_\omega(f)^2 d\mu_\beta =\brak{f,Q_\beta f}_{L^2(D)},\quad
    \int I^2_\omega(h)^2 d\mu_\beta =\brak{h,Q_\beta\otimes Q_\beta h}_{L^2(D^2)}. 
\end{equation}
If $h\in C^\infty_c(D^2)$ is smooth, and it vanishes on the diagonal, $h(x,x)=0 \ \forall x\in D$, then the double stochastic integral coincides with a pathwise duality coupling against a sample of $\omega$,
\begin{equation}\label{eq:equivalence_stoch_integral}
    I^2_\omega(h)=\brak{\omega\otimes\omega,h}
\end{equation}
(analogously to the fact that a smooth function can be integrated pathwise against Brownian motion, \emph{cf.} again \cite[Section 7.2]{Ja97}).

The Gaussian measure $\mu_\beta$ is mutually absolutely continuous with white noise on $D$: it holds (\emph{cf.} \cite[Lemma 2.3]{GrRo19})
\begin{equation*}
    d\mu_\beta(\omega)=\frac1{Z_\beta} \exp\pa{-\beta I^2_\omega(G)} d\mu_0(\omega),
\end{equation*}
where $\mu_0$ is space white noise, and the partition function $Z_\beta$ is well-defined for $\beta\geq 0$ thanks to the exponential integrability of double stochastic integrals
(\emph{cf.} \cite{Ja97} and again the discussion in \cite[Section 4]{GrRo19}).
The double stochastic integral of Green's function $G$ corresponds to a renormalized (\emph{i.e.} Wick ordered) version of the fluid energy.

\begin{remark}\label{rmk:formalinvariance}
    The existence of weak solutions to \eqref{eq:2deuler} preserving $\mu_\beta$ dates back to \cite{AlCr90}, and was first proved by means of a Galerkin approximation scheme. We refer to \cite{Grotto2020a,Flandoli2020} for generalizations of that result concerning Gaussian measures, and to \cite{Grotto2022} for invariant measures with non-Gaussian marginals.
\end{remark}

The following is the main result of \cite{GrRo19}, and it establishes a relation between Gibbs ensemble of vortices at inverse temperature $\beta\geq 0$ and the measure $\mu_\beta$, in the form of a limit theorem relying on the convergence of partition functions.

\begin{theorem}\cite[Theorem 4.2, Corollary 4.8]{GrRo19}\label{thm:mainGrottoRomito}
    Let $\beta\geq 0$, assume the neutrality condition \eqref{eq:neutrality} 
    and set $\bar{g}=$ $\int_D g(y, y) d y$. It holds:
    \begin{enumerate}
        \item $\lim _{N \rightarrow \infty} Z_{\beta}^N=e^{\beta \bar{g}} Z_{\beta}$;
        \item the sequence of $\mathcal{M}$-valued random variables $\omega^N \sim \mu_{\beta}^N$ converges in law on $H^s(D)$, any $s<-1$, to a random distribution $\omega \sim \mu_{\beta}$, as $N \rightarrow \infty$.
    \end{enumerate}
\end{theorem}
\noindent
The second item of \cref{thm:mainGrottoRomito} can be slightly improved as follows:
\begin{corollary}\label{cor:mainGrottoRomito}
Under the same assumptions and notation of \cref{thm:mainGrottoRomito}, $\omega^N\stackrel{law}{\rightharpoonup} \omega$ in $H^{-1-}(D).$
\end{corollary}
\begin{proof}
The sequence $\{\mu_{\beta}^N\}_{N\in\N}$ is tight on $H^{-1-}(D)$: thanks to \cref{thm:mainGrottoRomito}, for each $j\in\N,\ \epsilon>0$ there exists a compact set $K_{\epsilon,j}\subseteq H^{-1-\frac{1}{j}}(D)$ such that \begin{align*}
    \mu_{\beta}^N(K_{\epsilon,j}^c)<\frac{\epsilon}{2^j}\quad \forall N\in\N.
\end{align*}
Set $K_{\epsilon}=\cap_{j\in\N}K_{\epsilon,j}$. We have
\begin{align*}
    \mu_{\beta}^N(K_{\epsilon}^c)\leq \sum_{j\in\N}\mu_{\beta}^N(K_{\epsilon,j}^c)<\epsilon\quad \forall N\in\N.
\end{align*}
The set $K^{\epsilon}$ is relatively compact in the topology of $H^{-1-}(D)$. Hence $\{\mu_{\beta}^N\}_{N\in\N}$ is tight in $H^{-1-}(D)$. Therefore, by Prokhorov's theorem we can find a subsequence $\mu_{\beta}^{N_k}$ and a measure $\mu$ on $H^{-1-}(D)$ such that $\mu_{\beta}^{N_k}\rightharpoonup \gamma_{\beta}$. We are left to prove that $\gamma_{\beta}=\mu_{\beta}$. Indeed, since continuous bounded functions on $H^{-1-}(D)$ are also continuous bounded on $H^{s}$, $ s<-1$, we deduce that $\mu_{\beta}^{N_k}$ converges weakly to $\gamma_{\beta}$ in every $H^{s}$, therefore $\gamma_{\beta}=\mu_{\beta}$. By uniqueness of the limit point and the properties of the convergence in law, see \cite[Theorem 2.6]{Billinglsey}, it follows that the full sequence $\omega^{N}$ satisfies $\omega^N\stackrel{law}{\rightharpoonup} \omega$ in $H^{-1-}(D).$
\end{proof}

\subsection{Weak Vorticity Formulation of 2d Euler Equations}\label{subsec:weak_vorticity_formulation}

We can now discuss the notion of weak solution giving meaning to Euler's dynamics \eqref{eq:2deuler} in the case where solutions are measure-valued, as in the case of vortices, or have fixed time distribution $\mu_\beta$.

We introduce, for $N\in\N$ fixed, the following notation
\begin{align*}
    \triangle&:=\{(x,y)\in\bar{D}^2:\quad x=y\},\quad \partial\triangle:=\{(x,y)\in\bar{D}^2:\quad x=y\in \partial D\},\\
    \triangle_N&:=\{(x_1,\dots,x_N)\in D^N:\quad \exists i,j\quad s.t. \ i\neq j, (x_i,x_j)\in \triangle  \}    
\end{align*}
Moreover, denoting by
\begin{equation*}
    K^{\free}(x,y)=\frac{1}{2\pi}\frac{(x-y)^{\perp}}{\abs{ x-y}^2}=K(x,y)-\nabla^{\perp}_x g(x,y),\quad x,y\in D,
\end{equation*}
the Biot-Savart Kernel in the full space, we introduce for $\phi\in \mathcal{D}$
\begin{align*}
    H_\phi(x,y)&=
        \begin{cases}
                    0 & \text{if }(x,y)\in\triangle\\
                    \tfrac{1}{2} K^{\free}(x-y)\pa{\nabla\phi(x)-\nabla\phi(y)}& \text{if }(x,y)\notin\triangle
        \end{cases},\\
    h_\phi(x,y)&=
        \begin{cases}
                    0 & \text{if }(x,y)\in\partial\triangle\\
                    \tfrac{1}{2} \pa{\nabla^{\perp}_xg(x,y)\nabla \phi(x)+\nabla^{\perp}_yg(y,x)\nabla \phi(y)}& \text{if }(x,y)\notin\partial\triangle               
        \end{cases}.
\end{align*}
By symmetrizing integration variables in the standard weak formulation of \eqref{eq:2deuler}---an idea dating back to \cite{Delort1994,Schochet1995}---
one obtains the following:
\begin{lemma}
    Any solution $\omega(t)$ of \eqref{eq:2deuler} with initial datum $\omega(0)=\omega_0\in L^\infty (D)$, that is in the well-posedness class, satisfies the weak vorticity formulation:
    \begin{equation}\label{eq:weakvorticityformulation}
    \forall \phi\in \D,\, \forall t\in [0,T],\quad
    \brak{\omega_t,\phi}=\brak{ \omega_0,\phi}+\int_0^t \brak{ \omega_s\otimes \omega_s,H_\phi+h_\phi } ds.
    \end{equation}
\end{lemma}

The function $H_\phi$ is smooth outside of $ \triangle$, and a second order Taylor expansion of $\phi$ reveals that $H_\phi(x,y)\in \mathcal{B}_b(\bar{D}^2)$ is globally bounded, although it is discontinuous at $\triangle$.
As for $h_\phi$, thanks to the regularity of $\phi,g$, it is a smooth function outside $\partial \triangle$ 
and it is bounded for $(x,y)\rightarrow \partial \triangle$ since by maximum principle it holds 
    \begin{align*}
        \abs{ h_\phi(x,y)} \lesssim \frac{\norm{ \phi}_{C^2(\bar{D})}\pa{d(x,\partial D)+ d(y,\partial D) }}{d(x,\partial D)\vee d(y,\partial D)}\lesssim_{\phi,D} 1,
    \end{align*}
therefore $h_\phi\in \mathcal{B}_b(\bar{D}^2)$. Moreover,
\begin{equation*}
    \int_D h_\phi(x,x) dx=\int_D\nabla{\phi}(x)\nabla^{\perp}g(x,x)dx =\brak{ \nabla^{\perp}g, \nabla\phi}=0.
\end{equation*}

The weak vorticity formulation above allows point vortex flows as solutions, as first discussed in \cite{Schochet1996}.

\begin{lemma}\label{le:pvPDE}
    Given a solution $x_1(t),\dots,x_N(t)$ of \eqref{eq:pv} with intensities $\gamma_1,\dots,\gamma_N$, its empirical measure $\omega_t=\sum_{j=1}^N \gamma_j \delta_{x_j(t)}$ satisfies \eqref{eq:weakvorticityformulation} provided that the diagonal contribution in $\omega_t\otimes\omega_t$ is neglected in the coupling with $H_\phi$, that is interpreting
    \begin{equation*}
        \brak{ \omega_t\otimes \omega_t,H_\phi+h_\phi }
        =\sum_{j=1}^N\sum_{i< j}^N \gamma_i \gamma_j H_\phi(x_i,x_j)+\sum_{i,j=1}^N \gamma_i \gamma_j h_\phi(x_i,x_j).
    \end{equation*}
\end{lemma}
\noindent
Notice that in fact the regularity of $H_\phi,h_\phi$ is not sufficient for coupling them with measures, so the latter specification is essential.

\begin{proof}
    It holds, for $\phi\in \D$,
    \begin{align*}
    \frac{d}{dt}\brak{ \omega^N_t,\phi}
    &=\frac{1}{\sqrt{N}}\frac{d}{dt}\sum_{i=1}^N\xi_i\phi(x_i)\\ 
    & =\frac{1}{N}\sum_{i\neq j}\xi_i\xi_j K(x_i,x_j)\nabla\phi(x_i)+\frac{1}{N}\sum_{i=1}^N \xi_i^2 \nabla^{\perp}g(x_i,x_i)\nabla\phi(x_i)\\ 
    & =\frac{1}{N}\sum_{i\neq j}\xi_i\xi_j K^{\free}(x_i,x_j)\nabla\phi(x_i)+\frac{1}{N}\sum_{i,j=1}^N \xi_i\xi_j \nabla^{\perp}_xg(x_i,x_j)\nabla\phi(x_i)\\ 
    & =\brak{ \omega^N_t\otimes\omega^N_t,H_\phi+h_\phi},
\end{align*}
the last step following by symmetrization in the summation indices $i,j$ and the definitions of involved objects.
\end{proof}

\subsection{Main results}

The weak vorticity formulation allows to consider solutions of \eqref{eq:2deuler} that are only almost-surely defined with respect to a random initial data whose law is given by the Energy-Enstrophy measure, and that preserve the latter:

\begin{definition}\label{def:energyenstrophysol}
Given a probability space $(\Omega,\F,\PP)$, a measurable map, $\omega:\Omega\times[0,T]\rightarrow \mathcal{D}'$ 
with trajectories of class $C([0,T];H^{-1-}(D))$ is an Energy-Enstrophy solution of Euler equations if for each $t\in [0,T]$ the law of $\omega_t$ is $\mu_\beta$, and for every $\phi\in \mathcal{D}$ it holds $\PP$-almost surely
\begin{equation}\label{eq:weakvortstoch}
    \forall t\in [0,T],\quad I^1_{\omega_t}(\phi)-I^1_{\omega_0}(\phi)=\int_0^t I^2_{\omega_s}(H_\phi+h_\phi) ds.
\end{equation}
\end{definition}
In other words, the latter amounts to interpret the integrals in \eqref{eq:weakvorticityformulation} as stochastic integrals.
We stress again the fact that no external noise is involved in the dynamics, 
and that the probabilistic terminology is only used for its convenience in treating time-dependent objects defined on a measure space.
Even more importantly, Energy-Enstrophy solutions can never be Gaussian processes, notwithstanding the fact that the measure they preserve along time evolution is Gaussian:

\begin{lemma}\label{lem:nongaussian}
    Any Energy-Enstrophy solution in the sense of \cref{def:energyenstrophysol} either identically vanishes or it is not a Gaussian process.
\end{lemma}

\begin{proof}
    Assume by contradiction that an Energy-Enstrophy solution $\omega_t$ is a Gaussian process on $(\Omega,\F,\PP)$.
    In particular, any stochastic integral $I^1_{\omega_s}(f)$, $f\in L^2(D)$, 
    belongs to the first Wiener chaos associated to the Gaussian law of the whole process,
    and analogously any double stochastic integral $I^2_{\omega_s}(h)$, $h\in L^2(D^2)$, belongs to the second Wiener chaos.
    The right-hand side of \eqref{eq:weakvortstoch} is thus a linear combination of elements of the second Wiener chaos, thus it belongs to the latter itself.
    However, the left-hand side clearly belongs to the first Wiener chaos, so that the two sides of the equation \eqref{eq:weakvortstoch} 
    are orthogonal as random variables in $L^2(\Omega,\F,\PP)$. Therefore, either they vanish almost-surely or we contradict the Gaussian assumption.
\end{proof}

The notion of \cref{def:energyenstrophysol} is equivalent to that proposed in \cite{AlCr90}, 
in which existence was proved with a Galerkin approximation scheme:
we refer to \cite{Grotto2020a,GrottoPappalettera2021} for further discussions on this aspect.
An important contribution of \cite{Fla2018}, of which the forthcoming \cref{thm:limit} is a generalization,
was to provide an approximation of this kind of solutions by means of PV systems, thus further justifying its introduction. 
However, as already mentioned no notion or result of uniqueness is available.

\begin{theorem}\label{thm:limit}
    There exists a probability space $(\Omega,\mathcal{F},\PP)$ such that:
    \begin{enumerate}
    \item There exists a measurable map $\omega:\Omega\times[0,T]\rightarrow \mathcal{D}'(D)$ that is an Energy-Enstrophy solution in the sense of \cref{def:energyenstrophysol};
    \item There exists an increasing sequence $N_k\uparrow\infty$ and a sequence of 
    random initial configurations $\x^{N_k}_0=(x_0^{1,N_k},\dots, x_0^{N_k,N_k})$ for \eqref{eq:pv}, each with distribution $\nu^{N_k}_\beta$, such that the associated fluctuation fields
    \begin{equation*}
        \omega^{N_k}_t=\frac1{\sqrt{N_k}}\sum^{N_k}_{j=1} \xi_j \delta_{x^{j,N_k}(t)}
    \end{equation*}
    (defined as in \eqref{eq:defvort})
    converge $\PP$-almost surely in $C([0,T];H^{-1-}(D))$ to the process in item $(1)$.
\end{enumerate} 
\end{theorem}

\begin{remark}
    The proof by \cite{Fla2018} that Albeverio-Cruzeiro solutions at infinite temperature are limit of regular solutions is based on convergence of the point vortex dynamics and on the approximation of point vortices by regular solutions of Euler (vorticity localization) of \cite{Marchioro1993}.
    
    To conclude our program and prove that Albeverio-Cruzeiro solutions at every positive temperature $\beta>0$ are limit of regular solutions of Euler, it is sufficient to invoke analogous results of localization of vortices in a bounded domain \cite{DavDelMusWei2020,CecSei2021}.
\end{remark}

\begin{remark}
  In sight of \cref{lem:nongaussian}, our main \cref{thm:limit} above might be regarded as a \emph{non-Central Limit Theorem} for the fluctuation dynamics of the PV system. Determining the joint law of multi-time marginals of the limiting process is at present out of reach, but that would imply a selection principle among the possibly non-unique Energy-Enstrophy solutions of Euler's equations, which remains an important open problem in this context, \cite{Albeverio2008}.
\end{remark}

\section{Fluctuation Bounds}\label{sec:fluctuationbounds}

This Section is devoted to the proof of \cref{thm:spacebounds,thm:timebounds}.
The former is deduced from \cref{thm:mainGrottoRomito} by means of a combinatorial argument exploiting the neutrality condition \eqref{eq:neutrality}, which we detail in the forthcoming \cref{ssec:proofspacebound}. \cref{thm:timebounds} is the harder one, and it will require the Statistical Mechanics arguments outlined in the Introduction, which occupy the remainder of the Section.

\subsection{Functional Norms of Fluctuation Fields}\label{ssec:proofspacebound}

Let intensities $\xi_1,\xi_2,\dots,\xi_N\in\{\pm1\}$ satisfy \eqref{eq:neutrality}, and assume that the first $N/2$ equal $+1$ and the remaining $N/2$ take value $-1$.
For $k\in \set{1,\dots,N}$, we define $k'=k+N/2$ mod $N$, 
so that $\xi_k=-\xi_{k'}$.
Given natural numbers $0\leq m\leq n\leq N$, define
\begin{equation*}    
    \alpha_{m,n}\coloneqq\sum_{\substack{k_1,\dots,k_{n}=1\\\text{distinct}}}^N
    \xi_{k_1}\dots \xi_{k_m}
    =\sumast_{k_1,\dots,k_{n}=1}^N\xi_{k_1}\dots \xi_{k_m},
\end{equation*}
where the sum runs over all the possible ordered choices of distinct indices $k_1,\dots,k_{n}$ among $1,\dots,N$,
and the expression on the right-hand side is a shorthand notation for such sum to be used from now on.

\begin{lemma}\label{lem:alfas}
    For $0\leq m\leq n\leq N$ it holds
    \begin{equation}\label{eq:alfa}
        \alpha_{m,n}
        =\one_{2|m} (-1)^{m/2} \binom{N/2}{m/2}m!\frac{(N-m)!}{(N-n)!}.
    \end{equation}
\end{lemma}

\begin{proof}
    The case $m<n$ can be reduced to $m=n$ by
    \begin{equation*}
        \alpha_{m,n}=\frac{(N-m)!}{(N-n)!} \alpha_{m,m},
    \end{equation*}
    where the factor $\frac{(N-m)!}{(N-n)!}$ corresponds to the number of choices of indices that do not appear in the summand $\xi_{k_1}\dots \xi_{k_{m}}$.
    
    The map
    \begin{equation*}
        \phi(k_1,\dots,k_m)=(k_1',\dots,k_m'),
    \end{equation*}
    acts as a sign-reversing involution on the summands $\xi_{k_1}\dots \xi_{k_m}$ since $\xi_k=-\xi_{k'}$, therefore we can restrict the sum to fixed points of $\phi$.
    The latter consist of choices of indices $k_1,\dots,k_m$ for which $m$ is even and if $k$ belongs to the choice so does $k'$, therefore    
    \begin{multline*}
        \alpha_{m,m}=
        \sumast_{k_1,\dots,k_{m/2}=1}^{N/2}\xi_{k_1}\xi_{k_1'}\dots \xi_{k_{m/2}}\xi_{k_{m/2}'}\\
        =\sumast_{k_1,\dots,k_{m/2}=1}^{N/2} (-1)^{m/2}
        =(-1)^{m/2} \binom{N/2}{m/2}m!,
    \end{multline*}
    where $m!$ accounts for the possible orderings of factors.
\end{proof}

We recall again that for $\delta>0$ the Green function $G_{1+\delta}$ is bounded, that is there exists a constant independent of $x,y\in \bar{D}$ such that
\begin{equation}\label{eq:boundG1+delta}
    \abs{G_{1+\delta}(x,y)}\leq C_\delta.
\end{equation}
(as it follows by applying to $G_{1+\delta}(x,y)=(-\Delta)^{1+\delta/2}\delta_y(x)$ Morrey's inequality and the fact that $\norm{ \delta_y}_{H^{-1-\delta/2}}\lesssim_D 1$).

\begin{proof}[Proof of \cref{thm:spacebounds}]

We first rewrite the left hand side of \eqref{eq:lpbounds} and apply Cauchy-Schwarz inequality:
\begin{multline}\label{bound in space step 1}
    \int_{D^N} \norm{\frac1{\sqrt N}\sum_{i=1}^N\xi_i \delta_{x_i}}^p_{H^{-1-\delta}} d\nu_\beta^N\\
        =\frac{1}{Z_{\beta}^N}\int_{D^N} \norm{\frac1{\sqrt N}\sum_{i=1}^N\xi_i \delta_{x_i}}^p_{H^{-1-\delta}} e^{-\frac{\beta}{N}H(\x_N)}d\x_N \\
        \leq \frac{(Z_{2\beta}^N)^{1/2}}{Z_{\beta}^N} \pa{\int_{D^N} \norm{\frac1{\sqrt N}\sum_{i=1}^N\xi_i \delta_{x_i}}^{2p}_{H^{-1-\delta}}d\x_N}^{1/2}. 
\end{multline}
Since by \cref{thm:mainGrottoRomito} $(Z_{2\beta}^N)^{1/2}/Z_{\beta}^N\lesssim_{\beta} 1$, we are reduced to uniformly bound the integral 
\begin{gather}\nonumber
    I_N:= N^p\int_{D^N} \norm{\frac1{\sqrt N}\sum_{i=1}^N\xi_i \delta_{x_i}}^{2p}_{H^{-1-\delta}}d\x_N
    =\int_{D^N}\pa{\sum_{i,j=1}^N \xi_i \xi_j G_{\delta}(x_i,x_j)}^p d\x_N\\
    \label{eq:horrorsum}
    =\sum_{k_1,\dots,k_{2p}=1}^N \pa{\prod_{l=1}^{2p} \xi_{k_l}}\int_{D^N} \prod_{l=1}^{2p} G_{\delta}(x_{k_l},x_{k_{l+p}}) d\x_N.
\end{gather}

Given a choice of indices $\set{k_1,\dots,k_{2p}}$, define
\begin{equation*}
    \tilde j= \operatorname{argmin}\pa{j: \#\set{k_i=j}+\#\set{k_i'=j}\text{ is odd}}
    \in \set{1,\dots, N/2}
\end{equation*}
if such minimum exists, $\tilde j=\infty$ otherwise.
In other words, $\tilde j$ is such that there exist an odd number of indices $k_i,k_i'$ assuming that value, and it is the smallest such value. 
Notice that $\tilde j\leq N/2$ because if the condition is satisfied for an index $k_i>N/2$, then $k_i'\leq N/2$ also satisfies the condition over which we are minimizing.
Consider
\begin{gather*}
    \phi(k_1,\dots,k_{2p})=
    \begin{cases}
        (k_1,\dots,k_{2p}) &\tilde j=\infty\\
        (h_1,\dots,h_{2p}) &\tilde j<\infty
    \end{cases}
    ,\qquad 
    h_j=
    \begin{cases}
        k_j &k_j,k_j'\neq \tilde j\\
        \tilde j & k_j' = \tilde j\\
        \tilde j+N/2 & k_j = \tilde j\\
    \end{cases}.
\end{gather*}
The map $\phi$ acts on index choices by swapping an index $k_j\mapsto k_j'$ which occurs an odd number of times (and it is the identity if this is not possible):
this operation leaves invariant the integral factor $\int_{D^N} \prod_{l=1}^{2p} G_{\delta}(x_{k_l},x_{k_{l+p}}) d\x_N$, but it changes the sign of $\prod_{l=1}^{2p} \xi_{k_l}$ when it is not the identity. Hence, $\phi$ acts as a sign-reversing involution on the summands in \eqref{eq:horrorsum}, and we can thus restrict the sum to fixed points of $\phi$.

Since each summand of \eqref{eq:horrorsum} is uniformly bounded in $N$ by a constant depending only on $p,\delta$, $I_N$ is bounded from above by the number of fixed points of $\phi$, that is
\begin{equation*}
    \text{Fix}(\phi)=\set{k_1,\dots,k_{2p}:\,\tilde j=\infty}.
\end{equation*}
Elements of this set are choices of $2p$ indices such that the subset of those taking values $k,k'$ for a given $k$ has an even number of elements,
therefore every fixed point of $\phi$ can be obtained as follows.
Consider a weak composition $\alpha=(\alpha_1,\dots, \alpha_{N/2})$ of the integer $p$ of length $\ell(\alpha)=N/2$,
that is an array of non-negative integers $(\alpha_1,\dots, \alpha_{N/2})$ summing to $p$,
with $2\alpha_k$ representing the even number of indices equaling $k$ or $k'$ in a given fixed point.
Then, there are $2\alpha_k+1$ ways to decide how many of those indices take value $k$ or $k'$.
These choices determine a fixed point up to permutations of the whole index set $k_1,\dots,k_{2p}$: not all such permutations lead to different choices (\emph{e.g.} in the case where all indices coincide), but accounting for all of them provides an upper bound.
All in all, we can bound the number of fixed points of $\phi$ as follows:
\begin{multline*}
    \#\text{Fix}(\phi)\leq (2p)!
    \sum_{\substack{\alpha_1,\dots,\alpha_{N/2}\geq 0\\ \alpha_1+\cdots+\alpha_{N/2}=p }}
    \prod_{i=1}^{N/2} (2\alpha_i+1)
    \leq (2p)!
    \sum_{\substack{\alpha_1,\dots,\alpha_{N/2}\geq 0\\ \alpha_1+\cdots+\alpha_{N/2}=p }}
    \pa{\frac{N/2+2p}{N/2}}^{N/2}\\
    =(2p)!\binom{p+N/2-1}{N/2-1}\pa{\frac{N/2+2p}{N/2}}^{N/2}= O_p(N^p),
\end{multline*}
where the first passage is AM-GM inequality and the second one simply counts the number of compositions described above.
By the previous considerations, this implies that $I_N=O_{\delta,p}(N^p)$, from which the thesis directly follows.
\end{proof}

\subsection{Potential Splitting and Sine-Gordon Transformation}

The goal of the next Section is the proof of \cref{thm:timebounds}. As anticipated, the essential technical tool is the so-called Sine-Gordon Transformation: from a merely mathematical point of view this amounts to convert exponential functions of positive definite interaction functions to Gaussian integrals. 

We avoid to directly treat the singular interaction kernel of the PV system by splitting the potential into a regular, long range part converging to the full potential in the limit of the regularization parameter, and a singular, short range part to be dealt with as a negligible remainder. We follow the approach of \cite[Section 4.2]{GrRo19}.

Introduce the Green function $W_m$ of $m^2-\Delta$ with Dirichlet boundary conditions:
\begin{equation*}
W_m(x, y)=\frac{1}{2 \pi} K_0(m|x-y|)+w_m(x, y),
\end{equation*}
where the Bessel function $K_0$ provides a representation for the Green function of $m^2-\Delta$ in full space, and
\begin{equation*}
\begin{cases}
\left(m^2-\Delta\right) w_m(x, y)=0 & x \in D \\
w_m(x, y)=-\frac{1}{2 \pi} K_0(m|x-y|) & x \in \partial D
\end{cases}.
\end{equation*}
Decomposing $V_m=G-W_m$, $ v_m=g-w_m$, we can rewrite the Hamiltonian as
\begin{multline*}
H=\sum_{i<j}^N \xi_i \xi_j W_m\left(x_i, x_j\right)+\frac{1}{2} \sum_{i=1}^N \xi_i^2 w_m\left(x_i, x_i\right)\\
+\sum_{i<j}^N \xi_i \xi_j V_m\left(x_i, x_j\right)
+\frac{1}{2} \sum_{i=1}^N \xi_i^2 v_m\left(x_i, x_i\right)
:= H_{W_m}+H_{V_m},
\end{multline*}
but --as in \cite{GrRo19}-- we further rewrite the regular part $H_{V_m}$ by exploiting the neutrality condition \eqref{eq:neutrality} in terms of a zero-averaged kernel.
Introduce the following integral kernel on $D$:
\begin{equation*}
     V_m^0=M^*m^2(-\Delta(m^2-\Delta))^{-1}M,
\end{equation*}
$M$ being the space-averaging operator of \eqref{eq:meanoperator}, and
the inner inverse operator being taken under Dirichlet boundary conditions.
We also consider its counterpart in full space,
\begin{equation*}
    V_m^\free(x,y)=m^2(-\Delta(m^2-\Delta))^{-1}(x,y), \quad x,y\in \R^2,
\end{equation*}
which is a smooth, symmetric and translation-invariant function of two variables satisfying 
\begin{equation*}
    V_m^\free(0)=\frac{1}{2\pi}\log m+o(\log m), \quad m\to \infty.
\end{equation*}
Then it holds
\begin{equation*}
     H_{V_{m}}=\frac{1}{2} \sum_{i, j}^N \xi_i \xi_j V_m^0\left(x_i, x_j\right)-\frac{1}{2} \sum_{i=1}^N \xi_i^2 V_m^\free\left(x_i, x_i\right).
\end{equation*}
We now define the following regularized version of the Gaussian Free Field on $D$:
\begin{lemma}\label{lem:defF}
    Let $F_m$ be the centered Gaussian field on $D$ with covariance kernel $V_m^0$.
    There exists a version of $F_m(x)$ which is $\alpha$-H\"older for all $\alpha<\frac{1}{2}$, and for all $\alpha>0,\ p\geq 1$ and $m\rightarrow +\infty$, it holds
    \begin{align}\label{eq:Fpowerbound}
        \texpt{\norm{ F_m}_{L^p}^p}&\simeq_p (\log m)^{p/2},\\ \label{eq:Fexponentialbound}
        \texpt{e^{-\alpha\norm{ F_m}_{L^2}^2}}&\lesssim m^{-\frac{\alpha}{2\pi}}.
    \end{align}
    Moreover, it holds almost surely $\int_D F_m(x)dx=0$.
\end{lemma}

The first part of the statement was proved in \cite[Lemma 4.5]{GrRo19}:
H\"older regularity is a consequence of Kolmogorov continuity theorem, the $L^p$ and exponential bounds can be obtained by means of standard Gaussian computations.
The zero-average property follows from the fact that $\int_D F_m(x)dx=\brak{ F_m, 1}$ is by definition a centered Gaussian random variable, and
\begin{align*}
    \texpt{\brak{ F_m,1}^2}=\brak{ M^*m^2(-\Delta(m^2-\Delta))^{-1}M 1, 1}=0.
\end{align*}

The random field $F_m$ converges in law to the Gaussian Free Field as $m\to\infty$, as revealed by direct inspection of the covariance kernel, and we will employ $F_m$ it to represent the Boltzmannfaktor
\begin{equation*}
    e^{-\frac{\beta}{N} H_{V_m}}
      = e^{\frac\beta{2} V^\free_m(0)}
        \texpt{e^{i\frac{\sqrt{\beta}}{\sqrt{N}}\sum_j\xi_j F_m(x_j)}},
\end{equation*}
where $\tilde{\mathbb E}$ denotes the expectation with respect to law of $F_m$.
We refer to \cite{Stu1978} and \cite[Section 3]{Grotto2020decay} for a discussion on how this transformation allows to link the statistical mechanics of point vortices or 2d Coulomb gas to Sine-Gordon Euclidean Field Theory.

In order to reduce ourselves to consider the regular part of the interaction $V_m$ we will be relying on the following estimate on singular interactions.

\begin{lemma}\label{lem:singularpart}
    Assume the neutrality condition \eqref{eq:neutrality};
  then if $\beta\in (0,+\infty)$,
  \begin{equation*}
      e^{-\frac{\beta}{N} \overline{H}_{W_m}}\norm{e^{-\frac{\beta}{N}\pa{H_{W_m}-\overline{H}_{W_m}}}-1}_{L^s}\leq C_{\beta,s}\pa{\frac{N(\log m)^2}{m^2}}^{\frac{1}{s+2}}
  \end{equation*}
  where $\overline{H}_{W_m}=\int_{D^N} H_{W_m} d\x_N.$
\end{lemma}
\begin{proof}
    Preliminarily we observe that by \eqref{eq:neutrality} and the definition of $W_m$ and $w_m$, see also \cite[Proposition 4.6]{GrRo19}, introducing $r_m=\frac{2\log m}{m}$ it holds
    \begin{align}\label{uniform bound H_w}
        \abs{\frac{1}{N}\overline{H}_{W_m}}
        &\leq \abs{\frac{1}{2N}\sum_{i\neq j}\xi_i \xi_j \int_{D^{2}} W_m(x,y)dx dy}+\abs{\frac{1}{2N}\sum_{i=1}^N \xi_i^2 \int_D w_m(x,x) dx}\notag\\ 
        & = \frac{1}{2}\norm{ (m^2-\Delta)^{-1/2}1}_{L^2}^2+\frac{1}{2}\abs{\int_D w_m(x)dx}\notag\\ 
        &\lesssim \frac{1}{ m^2}+\int_{d(x,\partial D)\leq r_m} \log\pa{\frac{d(x,\partial D)}{r_m}} dx\notag\\ 
        & \lesssim \frac{1}{2m^2}+\frac{(\log m)^2}{ m^2} \lesssim \frac{(\log m)^2}{ m^2}.
    \end{align}
    Therefore $e^{-\frac{\beta}{N} \overline{H}_{W_m}}\lesssim 1$ uniformly in $m$ and $N$. Secondly we observe that since $H_{W_m}-\overline{H}_{W_m}$ has null space average, we can apply \cite[Lemma 4.2]{Grotto2020decay} to $e^{-\frac{\beta}{N} \pa{H_{W_m}-\overline{H}_{W_m}}}$ obtaining for every integer $n$ such that $s\leq 2n\leq s+2$,
    \begin{multline*}
     \norm{e^{-\frac{\beta}{N} \pa{H_{W_m}-\overline{H}_{W_m}}}-1}_{L^s}\\
     \leq  \norm{e^{-\frac{\beta}{N} \pa{H_{W_m}-\overline{H}_{W_m}}}-1}_{L^{2n}}  
     \leq \left(\int_{D^N}  e^{-2n\frac{\beta}{N} \left(H_{W_m}-\overline{H}_{W_m}\right)}-1  d\x_N\right)^{\frac{1}{2n}}\\  \lesssim e^{\frac{\beta}{N} \overline{H}_{W_m}}\abs{\int_{D^N}  e^{-2n\frac{\beta}{N} H_{W_m}}-1  d\x_N}^{\frac{1}{2n}} +\abs{ 1- e^{2n\frac{\beta}{N} \overline{H}_{W_m}}}^{\frac{1}{2n}} \vcentcolon J_1+J_2.
    \end{multline*}
The summand $J_1$ can be estimated as in \cite[Theorem 2.2]{Grotto2020decay}  obtaining \begin{align*}
    J_1\lesssim \pa{\frac{N(\log m)^2}{m^2}}^{\frac{1}{s+2}}.
\end{align*}
While $J_2$ can be simply bounded, up to some constants independent of $m, N$,  by $(\frac{ \log m}{m})^{\frac{2}{s+2}}$ exploiting Taylor's expansion and \eqref{uniform bound H_w}. Combining the two estimates the thesis follows.
\end{proof}

\subsection{Interaction Bounds}

Our asymptotic expansion of the partition function relative to the regular part of the interaction will rely on the following algebraic relation.
For $n,k\geq 1$, $\{A_{i,j}\}_{i,j\leq n}\subset \R$, we introduce the following notation:
\begin{equation*}
        \sideset{}{'}\sum_{j_1,\dots,j_k}^n A_{1,j_1}\cdots  A_{k,j_k}=\underbrace{
        \sum_{j_1=k}^n \sum_{j_2=k-1}^{j_1-1}\cdots \sum_{j_k=1}^{j_{k-1}-1}
        }_{k \text{ sums}} A_{1,j_1}\cdots  A_{k,j_k}.
   \end{equation*}

\begin{lemma}\label{lem:Algebraic_Lemma}
    It holds, for $J\geq 1$,
    \begin{multline*}
        \prod_{j=1}^n (a_j+b)= b^n+\sum_{k=1}^{J-1} b^{n-k} \sideset{}{'}\sum_{j_1,\dots,j_k}^n a_{j_1}\cdots a_{j_k}\\
        + \sideset{}{'}\sum_{j_1,\dots,j_J}^n a_{j_1}\cdots a_{j_{J-1}} \pa{ a_{j_J} b^{n-J+1-j_J} \prod_{i}^{j_J-1} (a_i+b) }.
    \end{multline*}
\end{lemma}
\noindent
The latter follows by iteration of the case $J=1$,
\begin{equation*}
    \prod_{j=1}^n (a_j+b)=b^n+\sum_{k=1}^n a_k b^{n-k}\pa{\prod_{j=1}^{k-1} (a_j+b)}.
\end{equation*}
We point out once again the similarity of the above expansion with the ones typical of Mayer cluster expansions, such as those of \cite{Brydges1978,Brydges1980}.

\begin{lemma}\label{l:further}
Let $\lambda, N\in\N,\ \lambda\geq 3,$ and assume the neutrality condition \eqref{eq:neutrality}. Let $F_m$ be the centred Gaussian random field with covariance
  function $V_m^0$, and for $r\in \{2,3,4\}$ consider distinct indices
  $j_1,\dots,j_r\in\{1,2,\dots,N\}$. 
Then, it holds   
\begin{multline}
\label{eq:exponential_asymptotic}
            \prod_{j\neq j_1,\dots,j_r}\int e^{i\frac{\sqrt{\beta}}{\sqrt{N}}\xi_j F_m(x_j)}dx_j
        = e^{-\beta\frac{N-r}{2N}\|F_m\|_{L^2}^2}\\
          + O\Bigl(N^{-1}\|F_m\|_{L^{3(\lambda-1)}}^{3(\lambda-1)}
            e^{-\beta\frac{N-r-\lambda+1}{2N}\|F_m\|_{L^2}^2}\Bigr)
         + O(N^{-\lambda/2}\norm{ F_m}_{L^{3\lambda}}^{3\lambda}).
  \end{multline}    
\end{lemma}
\begin{proof}
    In order to ease notation we introduce a new indexing $k\in \{1,\dots N-r\}$ for the elements of $\{1,\dots N\}\setminus\{j_1,\dots, j_r\}$. We also set
    \begin{equation*}
        E_k
      \vcentcolon=\int_De^{i\frac{\sqrt{\beta}}{\sqrt{N}}\xi_k F_m(x_k)}dx_k,
      \qquad
    B
      \vcentcolon=e^{-\frac\beta{2N}\|F_m\|_{L^2}^2}.
    \end{equation*}
  A fourth order Taylor expansion gives
    \begin{equation}\label{eq:single}
    E_k - B
      = -\frac{i\beta^{3/2}}{6 N^{3/2}}\xi_k\int F_m(x)^3dx
        + O(N^{-2}\|F_m\|_{L^4}^4).
  \end{equation} 
  Applying \cref{lem:Algebraic_Lemma} with $J=\lambda$, $n=N-r$, $b=B$, $a_j=E_j-B$ we get
    \begin{multline}\label{eq:step1exp}
        \prod_{j=1}^{N-r} E_j= B^{N-r}+\sum_{k=1}^{\lambda-1} B^{N-r-k} \sideset{}{'}\sum_{j_1,\dots,j_k}^{N-r} (E_{j_1}-B)\cdots (E_{j_k}-B)\\
        + \sideset{}{'}\sum_{j_1,\dots,j_\lambda}^{N-r} (E_{j_1}-B)\cdots (E_{j_{\lambda-1}}-B) \pa{ (E_{j_{\lambda}}-B) B^{N-r-\lambda+1-j_\lambda} \prod_{i}^{j_\lambda-1} E_i }.
        \end{multline}
  Since $\abs{ B}\leq 1,\ \abs{ E_n}\leq 1 $ for all $n\in \{1,\dots N-r\}$, it holds by \eqref{eq:step1exp} and \eqref{eq:single}
  \begin{multline}\label{eq:step2exp}
         \prod_{j=1}^{N-r} E_j= B^{N-r}+B^{N-r-1}\sum_{i=1}^{N-r}(E_i-B)\\
        +\sum_{k=2}^{\lambda-1} B^{N-r-k} \sideset{}{'}\sum_{j_1,\dots,j_k}^{N-r} (E_{j_1}-B)\cdots (E_{j_k}-B)+ O(N^{-\lambda/2}\norm{ F_m}_{L^{3\lambda}}^{3\lambda}).
    \end{multline}
  By \eqref{eq:neutrality} and \eqref{eq:single},
  \begin{multline*}
    \sum_{i=1}^{N-r}(E_i-B)\\
      = \frac{i\beta^{3/2}}{6 N^{3/2}}\Bigl(\sum_{\ell=1}^r\xi_{j_\ell}\Bigr)\int_D F_m(x)^3dx
        + O(N^{-1}\|F_m\|_{L^4}^4)
      = O(N^{-1}\|F_m\|_{L^4}^4).    
  \end{multline*}  
    Moreover, again by \eqref{eq:single}, if $2\leq k\leq \lambda-1$
    \begin{equation*}
     \sideset{}{'}\sum_{j_1,\dots,j_k}^{N-r} (E_{j_1}-B)\cdots (E_{j_k}-B)=O(N^{-k/2}\norm{ F_m}_{L^{3k}}^{3k}).
     \end{equation*}
  The last two estimates combined with \eqref{eq:step2exp} lead to the thesis:
  \begin{multline*}
      \prod_{j=1}^{N-r}E_j= e^{-\beta\frac{N-r}{2N}\|F_m\|_{L^2}^2}
          \\+ O\Bigl(N^{-1}\|F_m\|_{L^{3(\lambda-1)}}^{3(\lambda-1)}
            e^{-\beta\frac{N-r-\lambda+1}{2N}\|F_m\|_{L^2}^2}\Bigr)
        + O(N^{-\lambda/2}\norm{ F_m}_{L^{3\lambda}}^{3\lambda}).\qedhere
  \end{multline*}
\end{proof}

\begin{definition}\label{def:fwithindices}
    If $f:D^2\to \R$ be a bounded function, we set
    \begin{gather*}
        f_i=\int_D f(x,x) F_m(x)^i dx,\\
     f_{ij}=\int_{D^2} f(x,y)F_m(x)^i F_m(y)^j dx dy,\quad
    f^{ij}=\int_{D^2} f(x,x)f(x,y)F_m(x)^i F_m(y)^j dxdy,\\  
    f_{ijk}=\int_{D^3} f(x,x)f(y,z)F_m(x)^i F_m(y)^j F_m(z)^k dx dy dz,\\  
        f^{ijk}=\int_{D^3} f(x,y)f(x,z)F_m(x)^i F_m(y)^j F_m(z)^k dx dy dz.
    \end{gather*}
\end{definition}

\begin{lemma}\label{lem:weak_estimate_fF_terms}
    If $m=m(N)$ with at most polynomial growth, for $i,j,k,l\in\N, \vartheta>0,p\in (2,+\infty)$ it holds
    \begin{gather}
        \texpt{\abs{ f_i}^p}\lesssim_\dagger N^{\vartheta}|f|^p_{\infty,\triangle}, \notag \\
        \texpt{\abs{ f_{ij}}^p}\lesssim_\dagger N^{\vartheta}\norm{ f}_{L^p}^p, \quad
        \texpt{\abs{ f^{ij}}^p}\lesssim_\dagger N^{\vartheta} |f|^p_{\infty,\triangle}\norm{ f}_{L^p}^p, \notag\\
        \texpt{\abs{ f_{ijk}}^p}\lesssim_\dagger  N^{\vartheta}|f|^p_{\infty,\triangle}\norm{ f}_{L^p}^p, \quad 
    \texpt{\abs{ f^{ijk}}^p}\lesssim_\dagger  N^{\vartheta} \norm{ f}_{L^p}^{2p}, \notag\\
    \label{eq:Festimatefirstgroup}
    \texpt{\abs{ f_{ij}f_{kl}}^p}\lesssim_\dagger  N^{\vartheta}\norm{ f}_{L^{2p}}^{2p},
    \end{gather}
    where $\lesssim_\dagger$ denotes possible dependence on $i,j,k,l,p,\beta,\vartheta,D$.
    Moreover, using the same notation,
    for $p\in (2,+\infty)$ it holds:
    \begin{gather}\label{eq:Festimatesecondgroup}
        \texpt{|f_{10}|^{p}},\, \texpt{|f_{11}|^{p}}\lesssim_\dagger \norm{f}^p_{L^2},\quad 
        \texpt{|f^{011}|^p}\lesssim_\dagger \norm{f}^{2p}_{L^2}.
    \end{gather}
\end{lemma}
\begin{proof}
    The proof of formulas \eqref{eq:Festimatefirstgroup} consists in applying repeatedly H\"older's inequality and \cref{eq:Fpowerbound,eq:Fexponentialbound}. Letting $|f|^p_{\infty,\triangle}=\sup_{x\in \bar{D}}\abs{ f(x,x)}^p$, the first estimate follows from:
    \begin{equation*}
        \texpt{\abs{ f_i}^p}
        \leq |f|^p_{\infty,\triangle}\texpt{\norm{ F_m}^{ip}_{L^{ip}}}^{1/p}
        \lesssim_\dagger |f|^p_{\infty,\triangle} \pa{\log m}^{i/2} \lesssim N^{\vartheta} |f|^p_{\infty,\triangle}.
    \end{equation*}
    Denoting $\mathbf{F}_{a,b}=\texpt{\norm{ F_m}_{L^a}^b}$ for brevity, we have:
    \begin{align*}
        \texpt{\abs{ f_{ij}}^p}
        &\leq \norm{ f}_{L^p}^p \texpt{\norm{ F_m}_{L^{2ip/(p-1)}}^i\norm{ F_m}_{L^{2jp/(p-1)}}^j}\\
        &\leq \norm{ f}_{L^p}^p \mathbf{F}_{2ip/(p-1),2i}^{1/2}\mathbf{F}_{2jp/(p-1),2j}^{1/2}\\
        &\leq \norm{ f}_{L^p}^p
        \mathbf{F}_{2ip/(p-1),2ip/(p-1)}^{(p-1)/(2p)}\mathbf{F}_{2jp/(p-1),2jp/(p-1)}^{(p-1)/(2p)}\\
        &\lesssim_\dagger\norm{ f}_{L^p}^p (\log m)^{\frac{i+j}{2}}
        \lesssim N^{\vartheta}\norm{ f}_{L^p}^p;\\
        \texpt{\abs{ f_{ij}}^p}
        &\lesssim_\dagger |f|^p_{\infty,\triangle}\norm{ f}_{L^p}^p 
        \mathbf{F}_{2ip/(p-1),2i}^{1/2}\mathbf{F}_{2jp/(p-1),2j}^{1/2}
        \lesssim N^{\vartheta} |f|^p_{\infty,\triangle}\norm{ f}_{L^p}^p;\\
        \texpt{\abs{ f_{ijk}}^p}
        &\leq |f|^p_{\infty,\triangle}\norm{ f}_{L^p}^p 
        \texpt{\norm{ F_m}_{L^{3ip/(p-1)}}^i\norm{ F_m}_{L^{3jp/(p-1)}}^j\norm{ F_m}_{L^{3kp/(p-1)}}^k}\\ 
        &\leq |f|^p_{\infty,\triangle}\norm{ f}_{L^p}^p
        \mathbf{F}_{3ip/(p-1),3i}^{1/3}\mathbf{F}_{3jp/(p-1),3j}^{1/3}\mathbf{F}_{3kp/(p-1),3k}^{1/3}\\ 
        &\leq  |f|^p_{\infty,\triangle}\norm{ f}_{L^p}^p\mathbf{F}_{3ip/(p-1),3ip/(p-1)}^{(p-1)/(3p)}\\
        &\qquad\qquad\qquad\qquad \cdot
        \mathbf{F}_{3jp/(p-1),3jp/(p-1)}^{(p-1)/(3p)}\mathbf{F}_{3kp/(p-1),3kp/(p-1)}^{(p-1)/(3p)}\\ 
        &\lesssim_\dagger |f|^p_{\infty,\triangle}\norm{ f}_{L^p}^p \pa{\log m}^{\frac{i+j+k}{3}} 
        \lesssim N^{\vartheta}|f|^p_{\infty,\triangle}\norm{ f}_{L^p}^p;
    \end{align*}
    and arguing as for the latter one we have
    \begin{equation*}
        \texpt{\abs{ f^{ijk}}^p} \leq \norm{ f}_{L^p}^{2p}\texpt{\norm{ F_m}_{L^{3ip/(p-2)}}^i\norm{ F_m}_{L^{3jp/(p-2)}}^j\norm{ F_m}_{L^{3kp/(p-2)}}^k} \lesssim_\dagger N^{\vartheta}\norm{ f}_{L^p}^{2p}.
    \end{equation*}
    The last bound in \eqref{eq:Festimatefirstgroup} follows from a previous one,
    \begin{equation*}
        \texpt{\abs{ f_{ij}f_{kl}}^p} \leq \texpt{\abs{ f_{ij}}^{2p}}^{1/2}\texpt{\abs{ f_{kl}}^{2p}}^{1/2} \lesssim_\dagger N^{\vartheta}\norm{ f}_{L^{2p}}^{2p}.
    \end{equation*}

    The uniform estimates (in $m$) \eqref{eq:Festimatesecondgroup} are essentially due to the fact that $F_m$ converges in law, as $m\to \infty$, to a zero-average version of the \emph{Gaussian Free Field}.
    Indeed,
    \begin{equation*}\label{eq:Funiformbound}
        \texpt{\brak{ F_m,f}^2}
        =\brak{ m^2(-\Delta(m^2-\Delta))^{-1} M f,  Mf}
        \lesssim \norm{Mf}^2_{L^2}
    \end{equation*}
    (as it can be easily checked for instance by means of Fourier series).   
    Given $h,g\in C^\infty(\bar D)$, it holds
    \begin{multline*}
        \texpt{\abs{\int_{D^2}h(x)g(y)F_m(x)F_m(y)}^p}
        =\texpt{\abs{\int_{D}h(x)F_m(x)}^p \abs{\int_{D}g(y)F_m(y)}^p}\\
        \leq \texpt{\abs{\int_{D}h(x)F_m(x)}^{2p}}^{1/2} \texpt{\abs{\int_{D}g(y)F_m(y)}^{2p}}^{1/2}\\
        \lesssim_{p,\vartheta} \norm{h}_{L^2(D)}^p \norm{g}_{L^2(D)}^p
        \lesssim \norm{h\otimes g}^p_{L^2(D^2)},
    \end{multline*}
    the second step making use of the fact that inner integrals are Gaussian variables and \eqref{eq:Funiformbound}.
    Since the linear span of smooth separable functions $h(x)g(y)$ is dense in $L^2(D^2)$, the above estimates extends to the latter space proving the thesis, by linearity of $f\mapsto f_{01}$ and $f\mapsto f_{11}$. The estimate on $\texpt{|f^{011}|^p}$ follows along the same lines.
\end{proof}

\begin{proof}[Proof of \cref{thm:timebounds}.]
  First of all we notice that it is safe to assume that
  \begin{equation}\label{eq:fmeanzero}
    \int_{D\times D} f(x,y)dxdy
      = 0,
  \end{equation}
  since if $f$ is constant,
  \begin{equation*}
    \brak{ f,\omega^N\otimes\omega^N}
      = \frac{f}{N}\sum_{i, j=1}^N\xi_i\xi_j
      = 0.
  \end{equation*}
  It holds
  \begin{multline}\label{e:qtarget}
      \int_{D^N} \pa{\frac1N\sum_{i,j=1}^N \xi_i\xi_j f(x_i,x_j)}^2 d\nu^N_\beta(x_1,\dots, x_N)\\
      = \frac1{\zbgn N^2}\sum_{j_1, j_2,j_3,j_4}
        \int_{D^N}\xi_{j_1}\xi_{j_2}\xi_{j_3}\xi_{j_4}
        f(x_{j_1},x_{j_2})f(x_{j_3},x_{j_4})
        e^{-\beta H(\x_N)}d\x_N.
  \end{multline}
  We decompose $G = W_m + V_m$ as above: let $m=m(N)$ increasing with a polynomial rate to be specified later, and fix $\epsilon>0$.
  We first consider the part of the integral
  in \eqref{e:qtarget} corresponding to the singular
  part. This is bounded from above by
  \begin{equation}\label{e:Linfty1?}
    \frac{e^{-\frac{\beta}{N} \overline{H}_{W_m}}}{\zbgn }N^2\|f\|_{L^{2/\epsilon}}^2\norm{
      e^{-\frac{\beta}{N}H_{V_m}}}_{L^{1/\epsilon}}\norm{e^{-\frac{\beta}{N}( H_{W_m}-\overline{H}_{W_m})}-1}_{L^{1/(1-2\epsilon)}}
  \end{equation}
by H\"older's inequality. 
Now we observe that $\operatorname{inf}_{N\geq 1}Z_{\beta}^N>0$ thanks to \cref{thm:mainGrottoRomito},  by \cite[Proposition 2.8]{GrRo19} it holds
\begin{equation*}
    \sup_{N}\norm{e^{-\frac{\beta}{N}H_{V_m}}}_{L^{1/\epsilon}}\lesssim_{\beta,\epsilon} 1
\end{equation*}
and by \cref{lem:singularpart}, choosing $m=N^{\frac{1}{1-2\epsilon}+\frac{5}{2}}\log N$,
\begin{multline*}
          e^{-\frac{\beta}{N}\overline{H}_{W_m}}\norm{e^{-\frac{\beta}{N} \pa{H_{W_m}-\overline{H}_{W_m}}}-1}_{L^{1/(1-2\epsilon)}}
          \\ 
          \leq C_{\beta,\epsilon}\pa{\frac{N(\log m)^2}{m^2}}^{\frac{1}{\frac{1}{1-2\epsilon}+2}}=O\pa{\frac{1}{N^2}}.
\end{multline*} 
Therefore the singular part can be bounded uniformly in $N$ by $\norm{ f}_{L^{2/\epsilon}}^2$ up to some constant depending from $\epsilon, \beta$.

We now turn to the term that contains only the regular potential $H_{V_m}$.  
We will occasionally drop the subscript $m$ of $W_m,V_m, F_m$ to ease notation. 
We set $\lambda=\lceil (\frac{2}{1-2\epsilon}+5)\frac{\beta}{4 \pi}\rceil+1$, so that for each $p\in \R$ fixed it holds
\begin{align}\label{eq:asymptoticremainder}
        \frac{e^{\frac{\beta}{2}V_m^{\free}(0)}}{N^{\lambda/2}}\texpt{ \norm{ F_m}_{L^{3\lambda}}^{3\lambda p}}^{1/p}=o(1).
\end{align} 
We are left to study
\begin{align*}
     \frac{e^{-\frac{\beta}{N} \overline{H}_W}}{\zbgn N^2}\sum_{j_1, j_2,j_3,j_4}
        \int_{D^N}\xi_{j_1}\xi_{j_2}\xi_{j_3}\xi_{j_4}
        f(x_{j_1},x_{j_2})f(x_{j_3},x_{j_4})
        e^{-\frac{\beta}{N} H_V}d\x_N.
  \end{align*}
Thanks to the uniform bound guaranteed by \eqref{uniform bound H_w} we reduce ourselves to consider
  \begin{align*}
     \frac{1}{\zbgn N^2}\sum_{j_1, j_2,j_3,j_4}
        \int_{D^N}\xi_{j_1}\xi_{j_2}\xi_{j_3}\xi_{j_4}
        f(x_{j_1},x_{j_2})f(x_{j_3},x_{j_4})
        e^{-\frac{\beta}{N} H_V}d\x_N.
  \end{align*}
  We split the analysis into $7$ cases distinguishing the number of repeated indices.\\
  \emph{Case 1: $j_1=j_2=j_3=j_4$.}
  It holds
  \begin{multline}\label{e:Linfty?2}
    I_1=\frac{1}{\zbgn N^2}\sum_{j_1=1}^N
        \int_{D^N} f(x_{j_1},x_{j_1})^2 e^{-\frac{\beta}{N} H_V} d\x_N \\
        \lesssim\frac{\sup_{x\in \bar{D}}\abs{ f(x,x)}^2}{N}\norm{ e^{-\frac{\beta}{N} H_v}}_{L^1} \lesssim_{\beta} \frac{\sup_{x\in \bar{D}}\abs{ f(x,x)}^2}{N}
  \end{multline}
  by \cite[Proposition 2.8]{GrRo19}.\\
  \emph{Case 2: $j_1=j_2=j_3\neq j_4$.} We need to estimate
  \begin{align*}
      I_2&=\frac{1}{Z_{\beta}^N N^2}\sumast_{j_1,j_2} \xi_{j_1} \xi_{j_2} \int_{D^N}f(x_{j_1},x_{j_1})f(x_{j_1},x_{j_2}) e^{-\frac{\beta}{N} H_V}
        d\x_N.
  \end{align*}
By the Sine-Gordon transformation,
   \begin{multline}\label{eq:case2step1}
      I_2
      =\frac{\ebg}{\zbgn N^2}
      \sumast_{j_1,j_2}\xi_{j_1}\xi_{j_2}
          \Tilde{\mathbb{E}}\left[\Bigl(\prod_{\ell\neq j_1,j_2}\int_D e^{i\frac{\sqrt{\beta}}{\sqrt{N}}\xi_{j_\ell}F(x_\ell)}dx_\ell\Bigr)\right.
          \\ \times
      \braclose{
            \Bigl(\int_{D^2} f(x_{j_1},x_{j_1})f(x_{j_1},x_{j_2})\prod_{\ell=1}^2e^{i\frac{\sqrt{\beta}}{\sqrt{N}}\xi_{j_\ell}F(x_{j_\ell})}dx_{j_1}dx_{j_2}\Bigr)          
          }.
    \end{multline}
Our analysis proceeds with a third order Taylor expansion of the exponential factors in the integrand of \eqref{eq:case2step1}, combined with the following facts:
by \cref{lem:alfas} it holds
\begin{equation}\label{eq:alpha_m2}
    \sumast_{j_1,j_2}\xi_{j_1}\xi_{j_2} = -N,\qquad
    \sumast_{j_1,j_2}\xi_{j_1}= 0,
\end{equation} 
and, in the notation of \cref{def:fwithindices}, $f^{0n}=f^{m0}=0$ by \cref{eq:fnulldiagonal,eq:fmeanzero}. We thus obtain:
\begin{align*}
    \int_{D^2} f(x,x)&f(x,y) e^{i\frac{\sqrt{\beta}}{\sqrt{N}}(\xi_{j_1}F(x)+\xi_{j_2}F(y))}dxdy\\
    \qquad =  &-\frac{\beta}{N}\xi_{j_1}\xi_{j_2}\int_{D^2} f(x,x)f(x,y)F(x)F(y)dxdy\\ 
    &-\frac{i\beta^{3/2}}{2 N^{3/2}}\int_{D^2} f(x,x)f(x,y) \left[\xi_{j_1}F(x)^2F(y)+\xi_{j_2}F(x)F(y)^2\right]dxdy\\
    &+\int_{D^2} f(x,x)f(x,y) O\pa{\frac{\beta^2}{N^2}(\xi_{j_1}F(x)+\xi_{j_2}F(y))^4}dxdy,
\end{align*}
which implies, by \eqref{eq:alpha_m2},
\begin{multline}\label{eq:case2step2}
  \sumast_{j_1,j_2} \xi_{j_1}\xi_{j_2}\int_{D^2} f(x_{j_1},x_{j_1})f(x_{j_1},x_{j_2})
  \prod_{\ell=1}^2e^{i\frac{\sqrt{\beta}}{\sqrt{N}}\xi_{j_\ell}F(x_{j_\ell})}dx_{j_1}dx_{j_2}\\ 
  =-\beta(N-1) f^{11}+O\left(\norm{ F}_{L^8}^4\sup_{x\in \bar{D}}\abs{ f(x,x)} \norm{ f}_{L^2}\right).
\end{multline}       
Combining \eqref{eq:case2step1}, \eqref{eq:case2step2}, and \eqref{eq:exponential_asymptotic} for $r=2$ and $\lambda$ fixed above, we obtain:
\begin{multline*}
      I_2=\frac{\ebg}{\zbgn N^2} \Tilde{E}\left[\Bigl( e^{-\beta\frac{N-2}{2N}\|F\|_{L^2}^2} \right.\\ 
      + O\Bigl(N^{-1}\|F\|_{L^{3(\lambda-1)}}^{3(\lambda-1)} e^{-\beta\frac{N-1-\lambda}{2N}\|F\|_{L^2}^2}\Bigr) 
      + O(N^{-\lambda/2}\norm{ F}_{L^{3\lambda}}^{3\lambda})\Bigr)\\ 
        \times \left. \left(-\beta(N-1) f^{11}+O\pa{\norm{ F}_{L^8}^4\sup_{x\in \bar{D}}\abs{ f(x,x)} \norm{ f}_{L^2}}\right)\right].
\end{multline*}
Thanks to H\"older inequality, \cref{eq:Fpowerbound,eq:Fexponentialbound,eq:asymptoticremainder}, the asymptotic behavior of $V_m^{\free}(0)$ and \cref{lem:weak_estimate_fF_terms}, for ${\vartheta}=\epsilon$ we conclude that
\begin{multline*}
    I_2
    \lesssim_{\beta,\epsilon}\frac{\texpt{\abs{ f^{11}}^{1/\epsilon}}^{\epsilon}}{N}+\frac{\norm{ f}_{L^2}\sup_{x\in \bar{D}}\abs{ f(x,x)}\texpt{\norm{ F}_{L^8}^{4/\epsilon}}^{\epsilon}}{N^2}\\
    \lesssim_{\beta,\epsilon} \frac{\sup_{x\in\bar{D}}\abs{ f(x,x)}\norm{ f}_{L^{1/\epsilon}}}{N^{1-\epsilon}}+\frac{\norm{ f}_{L^2}\sup_{x\in \bar{D}}\abs{ f(x,x)}}{N^{2-\epsilon}}\notag\\
    \lesssim \frac{\sup_{x\in\bar{D}}\abs{ f(x,x)}\norm{ f}_{L^{1/\epsilon}}}{N^{1-\epsilon}}.
\end{multline*}
\emph{Case 3: $j_1=j_2\neq j_3=j_4$.}
  We need to estimate
  \begin{align*}
      I_3=\frac{1}{Z_{\beta}^N N^2}\sumast_{j_1,j_2} \int_{D^N}f(x_{j_1},x_{j_1})f(x_{j_2},x_{j_2}) e^{-\frac{\beta}{N} H_V}
        d\x_N.
  \end{align*}
  The forthcoming estimate follows closely the argument of Case 2:
  the Sine-Gordon reformulation reads
  \begin{multline*}
      I_3
      =\frac{\ebg}{\zbgn N^2}
      \sumast_{j_1,j_2}
          \tilde{\mathbb{E}}\left[\prod_{\ell\neq j_1,j_2}\int_D
              e^{i\frac{\sqrt{\beta}}{\sqrt{N}}\xi_{j_\ell}F(x_\ell)}dx_\ell\right.\\ 
              \times \braclose{
            \int_{D^2} f(x_{j_1},x_{j_1})f(x_{j_2},x_{j_2})
            e^{i\frac{\sqrt{\beta}}{\sqrt{N}}\xi_{j_1}F(x_{j_1})}
            e^{i\frac{\sqrt{\beta}}{\sqrt{N}}\xi_{j_2}F(x_{j_2})}
            dx_{j_1}dx_{j_2}   
          };
    \end{multline*}
    to which we apply again a Taylor expansion using the notation of \cref{def:fwithindices} (and keeping in mind \eqref{eq:fnulldiagonal}),    
\begin{align*}
  &\int_{D^2} f(x,x)f(y,y)
              e^{i\frac{\sqrt{\beta}}{\sqrt{N}}(\xi_{j_1}F(x)+\xi_{j_2}F(y))}dxdy\\
              &\quad =  \int_{D^2} f(x,x)f(y,y)\left[
              -\frac{\beta}{2N}(\xi_{j_1}F(x)+\xi_{j_2}F(y))^2\right]dxdy\\ 
              &\quad\quad+\int_{D^2} f(x,x)f(y,y) \left[-
              \frac{i\beta^{3/2}}{6N^{3/2}}(\xi_{j_1}F(x)+\xi_{j_2}F(y))^3\right]dxdy\\
              &\quad\quad+\int_{D^2} f(x,x)f(y,y) 
              O\pa{\frac{\beta^2}{N^2}(\xi_{j_1}F(x)+\xi_{j_2}F(y))^4}dxdy.
\end{align*}
The latter, together with \eqref{eq:alpha_m2} and \eqref{eq:fnulldiagonal},
allows to estimate the sum appearing in $I_3$:
Therefore by  we have \begin{align*}
    & \sumast_{j_1,j_2} \int_{D^2} f(x_{j_1},x_{j_1})f(x_{j_2},x_{j_2})\prod_{\ell=1}^2
              e^{i\frac{\sqrt{\beta}}{\sqrt{N}}\xi_{j_\ell}F(x_{j_\ell})}dx_{j_1}dx_{j_2}\\ 
    &\quad=-\sumast_{j_1,j_2}\xi_{j_1}\xi_{j_2} \frac{\beta}{N} f_1^2+\frac{i}{2}\frac{\xi_{j_1}+\xi_{j_2}}{ N^{3/2}}\beta^{3/2}f_2 f_1\\ 
    &\quad\quad\quad + \frac{N-1}{N}O\pa{\beta^2\norm{ F}_{L^4}^4\sup_{x\in \bar{D}}\abs{ f(x,x)}^2}\notag\\ 
    &\quad =\beta f_{1}^2+ O\pa{\norm{ F}_{L^4}^4\sup_{x\in \bar{D}}\abs{ f(x,x)}^2}.
\end{align*}
Combining the latter with \eqref{eq:exponential_asymptotic} for $r=2$ and $\lambda$ fixed above,
\begin{multline*}
          I_3
      =\frac{\ebg}{\zbgn  N^2}
          \Tilde{\mathbb{E}}\left[\Bigl( e^{-\beta\frac{N-2}{2N}\|F\|_{L^2}^2}\right.\\ 
          + O\Bigl(N^{-1}\|F\|_{L^{3(\lambda-1)}}^{3(\lambda-1)}
            e^{-\beta\frac{N-1-\lambda}{2N}\|F\|_{L^2}^2}\Bigr) + O(N^{-\lambda/2}\norm{ F}_{L^{3\lambda}}^{3\lambda})\Bigr)\\
            \braclose{
            \left( {\beta}f_{1}^2+ O\pa{\norm{ F}_{L^4}^4\sup_{x\in \bar{D}}\abs{ f(x,x)}^2} \right)          
          }.
\end{multline*}
Thanks to H\"older inequality, \autoref{lem:defF}, relation \eqref{eq:asymptoticremainder}, the asymptotic behavior of $V_m^{\free}(0)$, and \cref{lem:weak_estimate_fF_terms} for ${\vartheta}=\epsilon$, we conclude that
\begin{multline*}
    I_3 \lesssim_{\beta,\epsilon} \frac{1}{N^2}\texpt{\abs{ f_1}^{2/\epsilon}}^{\epsilon}+O\pa{\frac{1}{N^2}\texpt{\norm{ F}_{L^{4/\epsilon}}^{4/\epsilon}}^{\epsilon}\sup_{x\in \bar{D}}\abs{ f(x,x)}^2}\\  \lesssim_{\beta,\epsilon}  \frac{1}{N^2}\texpt{\abs{ f_1} ^{2/\epsilon}}^{\epsilon}+O\pa{\frac{1}{N^{2-\epsilon}} \sup_{x\in \bar{D}}\abs{ f(x,x)}^2}\\ 
    \lesssim_{\beta,\epsilon} \frac{1}{N^{2-\epsilon}} \sup_{x\in \bar{D}}\abs{ f(x,x)}^2.
\end{multline*}
   \emph{Case 4: $j_1=j_3\neq j_2=j_4$.} 
  We can bound
  \begin{align}\label{case 4 estimate 1}
      I_4=&\frac{1}{Z_{\beta}^N N^2}\sumast_{j_1,j_2} \int_{D^N} f(x_{j_1},x_{j_2})^2 e^{-\frac{\beta}{N} H_V}
        d\x_N \lesssim_{\beta} \norm{ f}_{L^{2/\epsilon}}^2 \norm{ e^{-\frac{\beta}{N} H_V}}_{L^{1/(1-\epsilon)}}
  \end{align}
  by H\"older's inequality. Thanks to \cite[Proposition 2.8]{GrRo19}, the right hand side of \eqref{case 4 estimate 1} can be estimate by $ \norm{ f}_{L^{2/\epsilon}}^2$ uniformly in $N$.\\
   \emph{Case 5: $j_1=j_2\neq j_3\neq j_4\neq j_1$.}  
  We need to estimate
  \begin{equation*}
      I_5=\frac{1}{Z_{\beta}^N N^2}\sumast_{j_1,j_2,j_3} \xi_{j_2} \xi_{j_3} \int_{D^N} f(x_{j_1},x_{j_1})f(x_{j_2},x_{j_3}) e^{-\frac{\beta}{N} H_V}
        d\x_N.
  \end{equation*}
  We follow again the strategy of Cases 2 and 3: first we write
  \begin{multline} \label{I_5 step 1}
      I_5
      =\frac{\ebg}{\zbgn N^2}
      \sumast_{j_1,j_2,j_3}
      \texptopen{\pa{\prod_{\ell\neq j_1,j_2,j_3}\int_De^{i\frac{\sqrt{\beta}}{\sqrt{N}}\xi_{j_\ell}F(x_\ell)}dx_\ell}}\\ \times
      \braclose{ \pa{\xi_{j_2}\xi_{j_3}\int_{D^3} f(x_{j_1},x_{j_1})f(x_{j_2},x_{j_3})\prod_{\ell=1}^3e^{i\frac{\sqrt{\beta}}{\sqrt{N}}\xi_{j_\ell}F(x_{j_\ell})}dx_{j_1}dx_{j_2}dx_{j_3}}  }.
  \end{multline}
  We then recall the appropriate cases of \cref{lem:alfas},
  \begin{equation}\label{eq:alphas_m3}
    \sumast_{j_1,j_2,j_3}\xi_{j_1}\xi_{j_2}\xi_{j_3}= 0,\qquad
    \sumast_{j_1,j_2,j_3}\xi_{j_1}\xi_{j_2}= -N(N-2), \qquad
    \sumast_{j_1,j_2,j_3}\xi_{j_1} = 0,
  \end{equation}   
  and observe that by symmetry and \cref{eq:fnulldiagonal,eq:fmeanzero}, it holds
  \begin{align}\label{eq:propertiesf3}
    f_{lmn}=f_{lnm},\quad f_{0mn}=0,\quad f_{l00}=0.
    \end{align}
    By Taylor's expansion and \eqref{eq:fnulldiagonal}, \eqref{eq:fmeanzero},
\begin{align*}
  &\int_{D^3} f(x,x)f(y,z)
    e^{i\frac{\sqrt{\beta}}{\sqrt{N}}(\xi_{j_1}F(x)+\xi_{j_2}F(y)+\xi_{j_3}F(z))}
    dxdydz\\
              &=  \int_{D^3} f(x,x)f(y,z)\left[
              -\frac{\beta}{2N}(\xi_{j_1}F(x)+\xi_{j_2}F(y)+\xi_{j_3}F(z))^2\right]dxdydz\\ 
              &\quad+\int_{D^3} f(x,x)f(y,z) \left[-
              \frac{i\beta^{3/2}}{6N^{3/2}}(\xi_{j_1}F(x)+\xi_{j_2}F(y)+\xi_{j_3}F(z))^3\right]dxdydz\\
              &\quad +\int_{D^3} f(x,x)f(y,z) 
              O\pa{\frac{\beta^2}{N^2}(\xi_{j_1}F(x)+\xi_{j_2}F(y)+\xi_{j_3}F(z))^4}dxdydz,
\end{align*}
therefore, by \eqref{eq:propertiesf3} and \eqref{eq:alphas_m3} we have
\begin{align*}
    & \sumast_{j_1,j_2,j_3} \xi_{j_2}\xi_{j_3}\int_{D^3} f(x_{j_1},x_{j_1})f(x_{j_2},x_{j_3})\prod_{\ell=1}^3
              e^{i\frac{\sqrt{\beta}}{\sqrt{N}}\xi_{j_\ell}F(x_{j_\ell})}dx_{j_1}dx_{j_2}dx_{j_3} \notag \\ 
    &=-\frac{i\beta^{3/2}}{2N^{3/2}}\sumast_{j_1,j_2,j_3}
    \pa{\xi_{j_3}f_{210}+\xi_{j_2}f_{201}+2\xi_{j_1}f_{111}}-\frac{\beta}{ N}\sumast_{j_1,j_2,j_3}(\xi_{j_1}\xi_{j_3}+\xi_{j_1}\xi_{j_2})f_{110} \\ 
    &\quad -\frac{i\beta^{3/2}}{2N^{3/2}}\sumast_{j_1,j_2,j_3}
    \pa{\xi_{j_1}\xi_{j_2}\xi_{j_3}f_{120}+\xi_{j_1}\xi_{j_2}\xi_{j_3}f_{102}}\\ & =\beta(N-2)f_{110}+ O\pa{N\norm{ F}_{L^8}^4\norm{ f}_{L^2}\sup_{x\in \bar{D}}\abs{ f(x,x)}}.
\end{align*}
Combining the latter with \eqref{I_5 step 1}, and \eqref{eq:exponential_asymptotic} for $r=3$ and $\lambda$ fixed above, we obtain
  \begin{multline*}
    I_5
    =\frac{\ebg}{\zbgn  N} \tilde{\mathbb{E}}\left[\Bigl(e^{-\beta\frac{N-3}{2N}\|F\|_{L^2}^2}\right. \\ 
    + O\Bigl(N^{-1}\|F\|_{L^{3(\lambda-1)}}^{3(\lambda-1)}
            e^{-\beta\frac{N-2-\lambda}{2N}\|F\|_{L^2}^2}\Bigr) + O(N^{-\lambda/2}\norm{ F}_{L^{3\lambda}}^{3\lambda})\Bigr)\\ 
    \braclose{
    \left(\beta\tfrac{N-2}{ N}f_{110}+ O\pa{\norm{ F}_{L^8}^4\norm{ f}_{L^2}\sup_{x\in \bar{D}}\abs{ f(x,x)}}\right)}.   
    \end{multline*} 
Thanks to H\"older inequality, \autoref{lem:defF}, relation \eqref{eq:asymptoticremainder}, the asymptotic behavior of $V_m^{\free}(0)$ and \cref{lem:weak_estimate_fF_terms} for $\vartheta=\epsilon$ we have
\begin{multline*}
    I_5 
    \lesssim_{\beta,\epsilon} 
    \tfrac{1}{N}\texpt{\abs{ f_{110}}^{1/\epsilon}}^{\epsilon}
    +\tfrac{1}{N}\texpt{\norm{ F}_{L^8}^{4/\epsilon}}^{\epsilon}\norm{ f}_{L^2}\sup_{x\in \bar{D}}\abs{ f(x,x)}\\ 
    \lesssim_{\beta,\epsilon} \frac{\norm{ f}_{L^{1/\epsilon}}\sup_{x\in \bar{D}}\abs{ f(x,x)}}{N^{1-\epsilon}}.
\end{multline*}

\emph{Case 6: $j_1=j_3\neq j_2\neq j_4\neq j_1$.}  
  We need to estimate
  \begin{equation*}
      I_6=\frac{1}{Z_{\beta}^N N^2}\sumast_{j_1,j_2,j_3}  \xi_{j_2} \xi_{j_3} \int_{D^N}f(x_{j_1},x_{j_2})f(x_{j_1},x_{j_3}) e^{-\beta H_V}
        d\x_N,
    \end{equation*}
which we rewrite as
  \begin{multline}\label{I_6 step 1}
      I_6
      =\frac{\ebg}{\zbgn  N^2}
      \sumast_{j_1,j_2,j_3}
          \tilde{\EE}\left[\pa{\prod_{\ell\neq j_1,j_2,j_3}\int_D
              e^{i\frac{\sqrt{\beta}}{\sqrt{N}}\xi_{j_\ell}F(x_\ell)}dx_\ell}\right.\\ 
        \braclose{
            \pa{\xi_{j_2}\xi_{j_3}\int_{D^3} f(x_{j_1},x_{j_2})f(x_{j_1},x_{j_3})\prod_{\ell=1}^3
              e^{i\frac{\sqrt{\beta}}{\sqrt{N}}\xi_{j_\ell}F(x_{j_\ell})}dx_{j_1}dx_{j_2}dx_{j_3}} 
          },
    \end{multline}
and we follow the same expansion argument of above, considering
\begin{align*}
\int_{D^3} f(x,y)&f(x,z)
              e^{i\sqrt{\beta}(\xi_{j_1}F(x)+\xi_{j_2}F(y)+\xi_{j_3}F(z))}dxdydz\\=&\int_{D^3} f(x,y)f(x,z)dxdydz \\ & +\int_{D^3} f(x,y)f(x,z)\left[
              i\frac{\sqrt{\beta}}{\sqrt{N}}(\xi_{j_1}F(x)+\xi_{j_2}F(y)+\xi_{j_3}F(z))\right]dxdydz\\ &+  \int_{D^3} f(x,y)f(x,z)\left[
              -\frac{\beta}{2N}(\xi_{j_1}F(x)+\xi_{j_2}F(y)+\xi_{j_3}F(z))^2\right]dxdydz\\ 
              &+\int_{D^3} f(x,y)f(x,z) \left[-
              \frac{i\beta^{3/2}}{6N^{3/2}}(\xi_{j_1}F(x)+\xi_{j_2}F(y)+\xi_{j_3}F(z))^3\right]dxdydz\\
              &+\int_{D^3} f(x,y)f(x,z) 
              O\left(\frac{\beta^2}{N^2}(\xi_{j_1}F(x)+\xi_{j_2}F(y)+\xi_{j_3}F(z))^4\right)dxdydz.
\end{align*}
In the notation of \cref{def:fwithindices}, it holds $ f^{lmn}=f^{lnm}$, hence by \eqref{eq:alphas_m3} we have
\begin{align}\label{I_6 step 2}
  & \sumast_{j_1,j_2,j_3} \xi_{j_2}\xi_{j_3}\int_{D^3} f(x_{j_1},x_{j_2})f(x_{j_1},x_{j_3})\prod_{\ell=1}^3
              e^{i\frac{\sqrt{\beta}}{\sqrt{N}}\xi_{j_\ell}F(x_{j_\ell})}dx_{j_1}dx_{j_2}dx_{j_3} \\ 
              &= \sumast_{j_1,j_2,j_3} \xi_{j_2}\xi_{j_3} f^{000}+i\frac{\sqrt{\beta}}{\sqrt{N}}\sumast_{j_1,j_2,j_3}\xi_{j_1}\xi_{j_2}\xi_{j_3}f^{100}+{\xi_{j_1}\xi_{j_3}}f^{010}+{\xi_{j_1}\xi_{j_2}}f^{010} \notag
              \\ 
              &\quad -\frac{\beta}{2N}\sumast_{j_1,j_2,j_3} {\xi_{j_2}\xi_{j_3}f^{200}}+{2\xi_{j_2}\xi_{j_3}f^{020}}+{2\xi_{j_1}\xi_{j_3}f^{110}}+{2\xi_{j_1}\xi_{j_2}f^{110}}+{2f^{011}} \notag\\ 
              &\quad-i\frac{\beta^{3/2}}{N^{3/2}}\sumast_{j_1,j_2,j_3}\frac{\xi_{j_3}f_{210}}{2}+\frac{\xi_{j_1}\xi_{j_2}\xi_{j_3}f_{120}}{2}+\frac{\xi_{j_2}f_{201}}{2}+\frac{\xi_{j_1}\xi_{j_2}\xi_{j_3}f_{102}}{2}+{\xi_{j_1}f_{111}}\notag\\ 
              &\quad-i\frac{\beta^{3/2}}{N^{3/2}}\sumast_{j_1,j_2,j_3}\frac{\xi_{j_1}\xi_{j_2}\xi_{j_3}}{6}f^{300}+\frac{\xi_{j_2}+\xi_{j_3}}{6}f^{030}+\frac{\xi_{j_2}f^{021}}{2}+\frac{\xi_{j_3}f^{021}}{2}\notag\\ 
              &\quad + \frac{(N-1)(N-2)}{N}O\pa{\norm{ F}_{L^{16}}^4\norm{ f}_{L^2}^2}\notag\\ 
              & =-{N(N-2)}f^{000}-2i\sqrt{\beta N }(N-2)f^{010}+\beta\frac{N-2}{2}(f^{200}+2f^{020}+4f^{110})\notag\\ \nonumber
              &\quad -\beta(N-1)(N-2)f^{011}+ O\pa{N\norm{ F}_{L^{16}}^4\norm{ f}_{L^2}^2}.
\end{align}
Combining \eqref{I_6 step 1}, \eqref{I_6 step 2}, and \eqref{eq:exponential_asymptotic} for $r=3$ and $\lambda$ fixed above, we obtain
\begin{multline*}
          I_6= \frac{\ebg}{\zbgn N}\tilde{\EE}\left[\Bigl(e^{-\beta\frac{N-3}{2N}\|F\|_{L^2}^2}\right.\\ 
          + O\pa{\tfrac1N\|F\|_{L^{3(\lambda-1)}}^{3(\lambda-1)}
            e^{-\beta\frac{N-2-\lambda}{2N}\|F\|_{L^2}^2}}
    + O(N^{-\lambda/2}\norm{ F}_{L^{3\lambda}}^{3\lambda})\Bigr)\\ 
            \times \Bigl( -\pa{N-2}f^{000}-2i\sqrt{\beta}
            \tfrac{N-2}{ \sqrt{N}}
            f^{010}+\beta\tfrac{N-2}{2 N}(f^{200}+2f^{020}+4f^{110})\\ 
    \braclose{ -\beta\tfrac{(N-1)(N-2)}{N}f^{011}+ O\pa{\norm{ F}_{L^{16}}^4\norm{ f}_{L^2}^2}  \Bigr)
          }.
\end{multline*}
Thanks to H\"older inequality, \autoref{lem:defF}, relation \eqref{eq:asymptoticremainder}, the asymptotic behavior of $V_m^{\free}(0)$ and \cref{lem:weak_estimate_fF_terms} for $\vartheta = \epsilon$ we conclude that
\begin{multline*}
    I_6
    \lesssim_{\beta,\epsilon} \abs{ f^{000}}
    +\texpt{\abs{ f^{011}}^p}^{1/p}
    +\frac1{N}{\texpt{\norm{ F}_{L^{16}}^{4/\epsilon}}^{\epsilon}\norm{ f}_{L^2}^2}
    +\frac1{N^{1/2}}{\texpt{\abs{ f^{010}}^{1/\epsilon}}^{\epsilon}}\\
    +\frac1{N}\pa{{\texpt{\abs{ f^{200}}^{1/\epsilon}}^{\epsilon}+\texpt{\abs{ f^{020}}^{1/\epsilon}}^{\epsilon}+\texpt{\abs{ f^{110}}^{1/\epsilon}}^{\epsilon}}}\\
    \lesssim_{\beta,\epsilon}\norm{ f}_{L^2}^2+\frac{\norm{ f}_{L^{1/\epsilon}}^2}{N^{{1/2}-\epsilon}},
\end{multline*}
where $p$ is an arbitrary value larger than $1$.\\
\emph{Case 7: $j_1,j_2,j_3,j_4$ are all different.}
  The final case follows once again the same expansion argument: consider
  \begin{multline}\label{eq:case7step1}
      I_7=\frac{1}{Z_{\beta}^N N^2 }
      \sumast_{j_1,j_2,j_3,j_4} \xi_{j_1} \xi_{j_2} \xi_{j_3} \xi_{j_4}
      \int_{D^N} f(x_{j_1},x_{j_2})f(x_{j_3},x_{j_4}) e^{-\frac{\beta}{N} H_V} d\x_N\\ 
        =\frac{\ebg}{\zbgn N^2}
      \sumast_{j_1,j_2,j_3,j_4}
          \tilde{\EE}\left[
            \pa{\xi_{j_1}\xi_{j_2}\int_{D^2} f(x_{j_1},x_{j_2})\prod_{\ell=1}^2
              e^{i\frac{\sqrt{\beta}}{\sqrt{N}}\xi_{j_\ell}F(x_{j_\ell})}dx_{j_1}dx_{j_2}}\right.\\ 
        \times\pa{\xi_{j_3}\xi_{j_4}\int_{D^2} f(x_{j_3},x_{j_4})\prod_{\ell=3}^4
              e^{i\frac{\sqrt{\beta}}{\sqrt{N}}\xi_{j_\ell}F(x_\ell)}dx_{j_3}dx_{j_4}}\\
        \times\braclose{\pa{\prod_{\ell\neq j_1,j_2,j_3,j_4}\int_D
              e^{i\frac{\sqrt{\beta}}{\sqrt{N}}\xi_{j_\ell}F(x_\ell)}dx_\ell}}.
\end{multline}
By Taylor's expansion,
\begin{align*}
  &\int_{D^2} f(x,y)
              e^{i\frac{\sqrt{\beta}}{\sqrt{N}}(\xi_{j_1}F(x)+\xi_{j_1}F(y))}dxdy\\&=  \int_{D^2} f(x,y)\left[
              i\frac{\sqrt{\beta}}{\sqrt{N}}(\xi_{j_1}F(x)+\xi_{j_2}F(y))-\frac{\beta}{2N}(\xi_{j_1}F(x)+\xi_{j_2}F(y))^2\right]dxdy\\ &+\int_{D^2} f(x,y)\left[
              -\frac{i\beta^{3/2}}{6N^{3/2}}(\xi_{j_1}F(x)+\xi_{j_2}F(y))^3+O\pa{\frac{\beta^2}{N^2}(\xi_{j_1}F(x)+\xi_{j_2}F(y))^4}\right]dxdy.
\end{align*}
By \eqref{eq:neutrality} it holds
\begin{gather}\label{e:sum3}
    \sumast_{j_1,j_2,j_3,j_4}
        \xi_{j_1}\xi_{j_2}\xi_{j_3}\xi_{j_4}
      = 3N(N-2),
      \qquad
    \sumast_{j_1,j_2,j_3,j_4}\xi_{j_1}\xi_{j_2}\xi_{j_3}
      = 0,\\ \nonumber
    \sumast_{j_1,j_2,j_3,j_4}\xi_{j_1}\xi_{j_2}
      = -N(N-2)(N-3),
      \qquad
    \sumast_{j_1,j_2,j_3,j_4}\xi_{j_1}
      = 0.
\end{gather}
Therefore, since $f_{mn}=f_{nm}$ thanks to the symmetry of $f$, we have
\begin{multline}\label{eq:case7taylor}
    \xi_{j_1}\xi_{j_2}\int_{D^2} f(x,y)
              e^{i\frac{\sqrt{\beta}}{\sqrt{N}}(\xi_{j_1}F(x)+\xi_{j_1}F(y))}dxdy\\ 
    =\frac{i\sqrt{\beta}(\xi_{j_1}+\xi_{j_2})}{\sqrt N}f_{10}-\beta\frac{\xi_{j_1}\xi_{j_2}}{ N}f_{20}-\frac{\beta}{N}f_{11}-\frac{i\beta^{3/2}(\xi_{j_1}+\xi_{j_2})}{6N^{3/2}}f_{30}\\ 
    -\frac{i\beta^{3/2}(\xi_{j_1}+\xi_{j_2})}{2N^{3/2}}f_{21} 
    +O\pa{\frac{\norm{ f}_{L^2}\norm{ F}_{L^{8}}^4}{N^2}}.
\end{multline}
And analogously for $\xi_{j_3}\xi_{j_4}\int_{D^2} f(x,y)
              e^{i\frac{\sqrt{\beta}}{\sqrt{N}}(\xi_{j_1}F(x)+\xi_{j_1}F(y))}dxdy$.
Combining \eqref{e:sum3} and \eqref{eq:case7taylor}  we obtain
\begin{align*}
    \sumast_{j_1,j_2,j_3,j_4}&
            \xi_{j_1}\xi_{j_2}\xi_{j_3}\xi_{j_4}\int_{D^4} f(x_{j_1},x_{j_2})f(x_{j_3},x_{j_4})\prod_{\ell=1}^4
              e^{i\frac{\sqrt{\beta}}{\sqrt{N}}\xi_{j_\ell}F(x_{j_\ell})}dx_{j_1}dx_{j_2}dx_{j_3}dx_{j_4}\notag \\ =&\sumast_{j_1,j_2,j_3,j_4}\bigg(-\beta\frac{(\xi_{j_1}+\xi_{j_2})(\xi_{j_3}+\xi_{j_4})}{N}f_{10}^2+\beta^2\frac{(\xi_{j_1}+\xi_{j_2})(\xi_{j_3}+\xi_{j_4})}{N^2}\pa{\frac{f_{10}f_{30}}{3}+f_{10}f_{21}} \notag\\ & +\beta^2\frac{\xi_{j_1}\xi_{j_2}\xi_{j_3}\xi_{j_4}}{N^2}f_{20}^2+\beta^2\frac{\xi_{j_1}\xi_{j_2}+\xi_{j_3}\xi_{j_4}}{N^2}f_{11}f_{20}+\frac{\beta^2}{N^2}f_{11}^2\notag \\ & +\beta^3\frac{(\xi_{j_1}+\xi_{j_2})(\xi_{j_3}+\xi_{j_4})}{N^3}\pa{\frac{f_{30}}{6}+\frac{f_{21}}{2}}^2\bigg)+O\pa{N^{3/2}\norm{ f}_{L^2}^2\norm{ F}_{L^{8}}^5}\notag\\ & =4\beta(N-2)(N-3)f_{10}^2-4\beta^2\frac{(N-2)(N-3)}{ N}\pa{\frac{f_{10}f_{30}}{3}+f_{10}f_{21}}\notag\\ & +\frac{3\beta^2(N-2)}{N}f^2_{20}-\frac{2\beta^2(N-2)(N-3)}{ N}f_{11}f_{20}\notag\\ &+\frac{\beta^2(N-1)(N-2)(N-3)}{N}f_{11}^2-4\beta^3\frac{(N-2)(N-3)}{ N^2}\pa{\frac{f_{30}}{6}+\frac{f_{21}}{2}}^2\notag \\ &+O\pa{N^{3/2}\norm{ f}_{L^2}^2\norm{ F}_{L^{8}}^5}.
\end{align*}
Using the latter in \eqref{eq:case7step1}, together with \eqref{eq:exponential_asymptotic} for $r=4$ and $\lambda$ fixed above, we obtain
\begin{align*}
      I_7
      &=\frac{\ebg}{\zbgn}  
          \Tilde{\EE}\left[
            \Bigg(4\beta\frac{(N-2)(N-3)}{ N^2}f_{10}^2-4\beta^2\frac{(N-2)(N-3)}{ N^3}\pa{\frac{f_{10}f_{30}}{3}+f_{10}f_{21}}\right.\notag\\ 
    & +\frac{3\beta^2(N-2)}{N^3}f^2_{20}-\frac{2\beta^2(N-2)(N-3)}{ N^3}f_{11}f_{20}\notag\\ &+\frac{\beta^2(N-1)(N-2)(N-3)}{N^3}f_{11}^2-4\beta^3\frac{(N-2)(N-3)}{N^4}\pa{\frac{f_{30}}{6}+\frac{f_{21}}{2}}^2\notag \\ 
    &+O\pa{\frac{\norm{ f}_{L^2}^2\norm{ F}_{L^{8}}^5}{N^{1/2}}}\Bigg)\times
            \Bigl(e^{-\beta\frac{N-4}{2N}\|F\|_{L^2}^2}\notag\\ 
    &\braclose{+ O\Bigl(N^{-1}\|F\|_{L^{3(\lambda-1)}}^{3(\lambda-1)}
            e^{-\beta\frac{N-3-\lambda}{2N}\|F\|_{L^2}^2}\Bigr)
            + O(N^{-\lambda/2}\norm{ F}_{L^{3\lambda}}^{3\lambda})\Bigr)
          }.
\end{align*}
Thanks to H\"older inequality, \autoref{lem:defF}, relation \eqref{eq:asymptoticremainder}, the asymptotic behavior of $V_m^{\free}(0)$ and \cref{lem:weak_estimate_fF_terms} for $\vartheta=\epsilon$ we have
\begin{multline*}
    I_7 
    \lesssim_{\beta,\epsilon} \texpt{\abs{ f_{10}}^{2p}}^{1/p}
    +\texpt{\abs{ f_{11}}^{2p}}^{1/p}\\
    +\frac{\texpt{\abs{ f_{10}f_{30}}^{1/\epsilon}}^{\epsilon}
    +\texpt{\abs{ f_{10}f_{21}}^{1/\epsilon}}^{\epsilon}}{N} 
    +\frac{\texpt{\abs{ f_{11}f_{20}}^{1/\epsilon}}^{\epsilon}}{N}\\
    +\frac{\texpt{\abs{ f_{20}}^{2/\epsilon}}^{\epsilon}
    +\texpt{\abs{ f_{30}}^{2/\epsilon}}^{\epsilon}
    +\texpt{\abs{ f_{21}}^{2/\epsilon}}^{\epsilon}}{N^2} 
    +\frac{\norm{ f}_{L^2}^2}{N^{1/2}}\texpt{\norm{ F}_{L^8}^{5p}}^{1/p}\\ 
    \lesssim_{\beta,\epsilon}\norm{ f}_{L^2}^2
    +\frac{\norm{ f}_{L^{2/\epsilon}}^2}{N^{1-\epsilon}}
    +\frac{\norm{ f}_{L^{2/\epsilon}}^2}{N^{2-\epsilon}}
    +\frac{\norm{ f}_{L^2}^2}{N^{1/2-\epsilon}}\\ 
    \lesssim \norm{ f}_{L^2}^2
    +\frac{\norm{ f}_{L^{2/\epsilon}}^2}{N^{1-\epsilon}}
    +\frac{\norm{ f}_{L^2}^2}{N^{1/2-\epsilon}}
\end{multline*}
where $p$ is an arbitrary value larger than $1$.
\end{proof}

\section{A Limit Theorem for Point Vortex Flows}
This section is devoted to the proof of \cref{thm:limit}. 
We need to exhibit a sequence of initial configurations such that the associated fluctuation fields
\begin{equation*}
    \omega^N_t=\frac{1}{\sqrt{N}}\sum_{i=1}^N \xi_i \delta_{x_t^{i,N}},
\end{equation*} 
(defined as in \eqref{eq:defvort})
converge to an Energy-Enstrophy solution of Euler equations in the sense of \cref{def:energyenstrophysol}. Therefore, we introduce a sequence $\{\xi_i\}_{i\in\N}$ such that $\xi_{i}=\pm 1$ for each $i$ and, for each $N\in \mathbb{N}$ even, the family $\{\xi_i\}_{i=1,\dots N}$ satisfies condition \eqref{eq:neutrality}. Then,  we introduce the random initial configurations $\x_0^N=(x_0^{1,N},\dots, x_0^{N,N})$ distributed according to the Canonical Gibbs ensemble $\nu_{\beta}^N$. 

As discussed in \cref{subsec:weak_vorticity_formulation,subseq:PVdynamics}, PV dynamics $\x^{N}_t$ starting from the initial condition $\x^N_0$ are well defined and preserve the measure $\nu_{\beta}^N$, and 
the fluctuation field has fixed-time law
\begin{equation*}
    \LL(\omega^N_t)=\LL(\omega^N_0)=\mu_{\beta}^N, \quad t\geq 0.
\end{equation*}

\subsection{Compactness}\label{subsec:compactness} We will rely on:

\begin{theorem}\cite{Simon1986}\label{compactness criteria Simon}
    If $X,\ B,\ Y$ are separable Banach spaces such that
\begin{align*}
    X\stackrel{c}{\hookrightarrow}B\hookrightarrow Y,
\end{align*}
where $\stackrel{c}{\hookrightarrow}$ means compact embedding.
    Assume that there exists $\theta \in (0,1)$ such that
    \begin{align*}
        \norm{ v}_{B}\leq M\norm{ v}_X^{1-\theta}\norm{ v}_Y^{\theta} \quad \forall\, v\in X .
    \end{align*}
    Let $F$ be bounded in $W^{s_0,r_0}(0,T;X)\cap W^{s_1,r_1}(0,T;Y),\ r_0,\ r_1\in [1,+\infty]$. Define
      $$ s_{\theta}=(1-\theta)s_0+\theta s_1,\quad \frac{1}{r_{\theta}}=\frac{1-\theta}{r_0}+\frac{\theta}{r_1},\quad s_\ast=s_{\theta}-\frac{1}{r_{\theta}}.$$
    If $s_\ast<0$ then $F$ is relatively compact in $L^p(0,T;B)$ for each $p<-1/s_\ast$, and if $s_\ast>0$ then $F$ is relatively compact in $C([0,T];B)$.
\end{theorem}
The embedding assumptions are satisfied in the following particular case,
thanks to the arguments in \cite[Chapters 11, 12]{LioMag} and \cite[Lemma 26]{Fla2018}.
\begin{corollary}\label{cor:simon}
	Let $\delta>0, \alpha>5$. If a family of functions
	\begin{equation*}
	\{v_n\}\subset L^q(0,T;H^{-1-\delta/2}(D))\cap W^{1,2}(0,T;H^{-\alpha}(D))
	\end{equation*}
	is bounded for $q>\frac{2(\alpha-1-\delta)}{\delta}$, then it is relatively compact in 	$C([0,T];H^{-1-\delta}(D))$.
	
As a consequence, if a sequence of stochastic processes $u^n:[0,T]\rightarrow H^{-1-\delta}(D)$, $n\in\N$, 
	defined on a probability space $(\Omega,\mathcal{F},\PP)$ is such that, for any $q>\frac{2(\alpha-1-\delta)}{\delta}$,
	there exists a constant $C_{T,\delta,q}$ for which
	\begin{equation}\label{eq:momentscondition}
	\sup_n \EE\left[\norm{ u^n}^q_{ L^q(0,T;H^{-1-\delta/2})}+\norm{ {u^n}}^2_{W^{1,2}(0,T;H^{-\alpha})}\right]
	\leq C_{T,\delta,q},
	\end{equation} 
	then the laws of $u^n$ on $C([0,T];H^{-1-\delta})$ are tight.
\end{corollary}

By the definitions in \cref{sec:sobolevspace} and \cref{cor:simon}, we have:
\begin{corollary}\label{cor:simon 2}
	Let $\alpha>3$. If a family of functions
	\begin{equation*}
	\{v_n\}\subset \mathcal{Y}:=L^{\infty-}(0,T;H^{-1-}(D)) \cap W^{1,2}(0,T;H^{-\alpha}(D))
	\end{equation*}
	is bounded, then it is relatively compact in $\XX$.
	
As a consequence, if a sequence of stochastic processes $u^n:[0,T]\rightarrow H^{-1-}(D)$, $n\in\N$, 
	defined on a probability space $(\Omega,\mathcal{F},\PP)$ is such that
	there exists a constant $C_{T}$ for which
	\begin{equation}\label{eq:momentscondition bicor}
	\sup_n \EE\left[d_{L^{\infty-}(H^{-1-})}(u^n_\cdot,0)+\norm{ u^n}_{W^{1,2}(0,T;H^{-\alpha})}^2\right]
	\leq C_{T},
	\end{equation} 
	then the laws of $u^n$ on $\XX$ are tight.    
\end{corollary}

Combining \cref{Compactess in space}, \cref{Compactness in time} below and \cref{cor:simon 2} above we obtain the tightness of $\LL(\omega^N)$ in $\XX$.
\begin{proposition}\label{Compactess in space}
Assuming $\beta\in [0,+\infty), \ \delta>0,\  q>2$ it holds
\begin{align*}
    \sup_N \int \norm{ \omega^N}^q_{ L^q(0,T;{H}^{-1-\delta/2})} \ d \LL(\omega^N_0)
	\leq C_{T,\beta,\delta,q}.
\end{align*}
\end{proposition}
\begin{proof}
The result follows from the invariance of the law of $\omega^N_t$ and \cref{thm:spacebounds}. Indeed it holds    
\begin{multline*}
    \int \int_0^T \norm{\omega^N_s}_{H^{-1-\delta/2}}^q\ ds d\LL(\omega^N_0) =\int_0^T \int\norm{\omega^N_s}_{H^{-1-\delta/2}}^q d\mu_{\beta}^N ds\\  = \int_0^T \int_{D^N} \norm{\frac1{\sqrt N}\sum_{i=1}^N\xi_i \delta_{x_i}}^q_{H^{-1-\delta/2}} d\nu_\beta^N ds \leq TC_{\beta,\delta,q}.
\end{multline*}
Since the inequality above is uniform in $N$ the thesis follows.
\end{proof}
\begin{proposition}\label{Compactness in time}
Assuming $\beta\in  [0,+\infty),\ \alpha>5$ it holds
\begin{align*}
    \sup_N \int\norm{ {\omega^N}}^2_{W^{1,2}(0,T;H^{-\alpha})}\ d\LL(\omega^N_0)
	\leq C_{T,\beta,\alpha,D}.
\end{align*}
\end{proposition}
\begin{proof}
    Thanks to \cref{Compactess in space} it is enough to show that 
    \begin{align*}
        \sup_N \int\int_0^T \norm{ \partial_t \omega^N_t}^2_{H^{-\alpha}}dt d\LL(\omega^N_0)
	\leq C_{T,\beta,\alpha,D}.
    \end{align*}
    This relation holds thanks to \cref{thm:timebounds}, the fact that $\mathcal{D}$ is dense in $H^{\alpha}_0$ and the embedding of $H^{\alpha}_0(D)\hookrightarrow H^4_0(D)$ is of Hilbert-Schmidt. In order to prove this claim first we recall that for each $\phi\in \D$ it holds
    \begin{align*}
        \langle \omega^N_t,\phi\rangle-\langle \omega^N_0,\phi\rangle&=\int_0^t \langle \omega^N_s\otimes \omega^N_s, H_\phi+h_\phi\rangle ds \quad \PP-a.s.
    \end{align*} Since the function $s\rightarrow \langle \omega^N_s\otimes \omega^N_s, H_\phi+h_\phi\rangle $ is $\PP$-a.s. continuous (the trajectories of the point-vortices are continuous and never touch the diagonal or the boundary of $D$) then  the function $\langle \omega^N_t,\phi\rangle$ is continuously differentiable and $\partial_t\langle \omega^N_s,\phi\rangle=\langle \omega^N_t\otimes \omega^N_t, H_\phi+h_\phi\rangle$.
    Secondly, we introduce $\{e_k\}_{k\in\N}$, an orthonormal basis of $H^{\alpha}_0(D)$ made by functions in $\D$, and $\mathcal{R}$, the Riesz map between $H^{\alpha}_0(D)$ and $H^{-\alpha}(D),$ i.e. given $g\in H^{-\alpha}(D)$ then $\mathcal{R}g$ satisfies
     \begin{align*}
        \langle \mathcal{R}g,\phi\rangle_{H^{\alpha}_0,H^{\alpha}_0}=\langle g,\phi\rangle_{H^{-\alpha},H^{\alpha}_0}\quad \forall \phi\in H^{\alpha}_0(D),\quad \norm{ \mathcal{R}g}_{H_0^{\alpha}}^2=\norm{ g}_{H^{-\alpha}}^2.
    \end{align*}
    Therefore it holds, 
    \begin{multline}\label{chain equalities}
        \norm{ \partial_t\omega^N_t}_{H^{-\alpha}}^2=\norm{ \mathcal{R}\partial_t\omega^N_t}_{H^{\alpha}_0}^2  =\sum_{k\in \N} \langle \mathcal{R}\partial_t\omega^N_t,e_k\rangle_{H^{\alpha}_0, H^{\alpha}_0}^2\\ =\sum_{k\in \N}\langle \partial_t\omega^N_t,e_k\rangle_{H^{-\alpha},H^{\alpha}_0}^2=\sum_{k\in \N}\langle \omega^N_t\otimes\omega^N_t, H_{e_k}+h_{e_k}\rangle^2 .
    \end{multline}
    We recall that if $\phi\in \D$ then $H_\phi+h_\phi$ satisfies the assumptions of \cref{thm:timebounds}. Combining \cref{thm:timebounds} and \eqref{chain equalities} the thesis follows since it holds: 
    \begin{multline*}
        \int\norm{ \partial_t\omega_t^N}_{H^{-\alpha}}^2\ d\LL(\omega^N_0)
        =\sum_{k\in\N}\int\langle \omega^N_t\otimes\omega^N_t, H_{e_k}+h_{e_k}\rangle^2\ d\mu^{N}_{\beta}\\ 
        \lesssim_{\beta} \sum_{k\in\N}\norm{ H_{e_k}+h_{e_k}}_{L^{\infty}}^2
        \lesssim \sum_{k\in\N}\norm{ e_k}_{C^2(\bar{D})}^2 \lesssim \sum_{k\in\N}\norm{ e_k}_{H^4_0}^2 =C_{\alpha}<+\infty.\qedhere
    \end{multline*}
\end{proof}

\subsection{Approximating Interaction Functions}\label{subse:ApproxIntFun}

We will need smooth approximations of $H_\phi$ and $h_\phi$. Let $r>0$ fixed, $\varphi:\R^2\rightarrow [0,1]$ a smooth function supported by $B_{0}(r)$, taking value $1$ on $B_0(r/2)$, and introduce $\varphi_{\delta}$ such that $\varphi_{\delta}(x)=\varphi(x/\delta)$. Secondly, fix a smooth function $\hat{\varphi}:\R\rightarrow [0,1]$ supported by $(-r,r)$, taking value $1$ on $[-r/2,r/2]$, and introduce  $\hat{\varphi}_{\delta}:\bar{D}\rightarrow\R$ such that $\hat{\varphi}_{\delta}(x)=\hat{\varphi}(\frac{d(x,\partial D)}{\delta})$. We then define:
\begin{align*}
    &H_\phi^{\delta}:\bar{D}^2\rightarrow \R^2,\quad H_\phi^{\delta}(x,y)=H_\phi(x,y)(1-\varphi_{\delta}(x-y))(1-\hat{\varphi}_{\delta}(x))(1-\hat{\varphi}_{\delta}(y)),\\
     & h_\phi^{\delta}:\bar{D}^2\rightarrow \R^2,\quad h_\phi^{\delta}(x,y)=h_\phi(x,y)(1-\varphi_{\delta}(x-y))(1-\hat{\varphi}_{\delta}(x))(1-\hat{\varphi}_{\delta}(y)).
\end{align*}
Since $H_\phi$ and $h_\phi$ are smooth outside $\triangle$, then $H_\phi^{\delta}$ and $h_\phi^{\delta}$ are smooth on $\bar{D}^2$. Something more can be said on $H_\phi^{\delta}$ and $h_\phi^{\delta}$: themselves and their derivatives of each order are smooth and equal to $0$ on $\triangle\cup \pa{\partial D\times D}\cup \pa{ D\times \partial D}.$ Moreover, by definition
\begin{align*}
   \abs{ H_\phi^{\delta}(x,y)} \leq \abs{ H_\phi(x,y)}, \quad \abs{ h_\phi^{\delta}(x,y)} \leq \abs{ h_\phi(x,y)},\\
   H_\phi^{\delta}(x,y)\rightarrow H_\phi(x,y)\quad a.e.\ (x,y)\in\bar{D}^2\\ h_\phi^{\delta}(x,y)\rightarrow h_\phi(x,y)\quad a.e.\ (x,y)\in\bar{D}^2.
\end{align*}
Therefore 
\begin{align*}
    \sup_{\delta>0}\norm{ H_\phi^{\delta}}_{C(\bar{D}^2)}\leq \sup_{(x,y)\in\bar{D}^2}\abs{ H_\phi(x,y)}<+\infty,\\
     \sup_{\delta>0}\norm{ h_\phi^{\delta}}_{C(\bar{D}^2)}\leq \sup_{(x,y)\in\bar{D}^2}\abs{ h_\phi(x,y)}<+\infty,\end{align*}
     and by Lebesgue's dominated convergence Theorem, as $\delta\to 0$,
     \begin{equation*}
     \norm{ H^{\delta}_{\phi}-H_\phi}_{L^p}\rightarrow 0,\quad
     \norm{ h^{\delta}_{\phi}-h_\phi}_{L^p}\rightarrow 0,
    \end{equation*}
for all $p\in [1,+\infty)$.

\subsection{Passage to the Limit}\label{subsec: passage to the limit}
Thanks to the results of \Cref{subsec:compactness}, the sequence $Q^N=\LL(\omega^N)$ is tight on $\XX$. 
Therefore by Prohorov's theorem we can find a subsequence $N_k$ and a Borel probability measure $Q$ on $\XX$ such that $Q^{N_k}\rightharpoonup Q$. 
Under $Q$, the identity process $\XX\ni \omega \mapsto (\omega_t)_{t\in T}$ is such that all single time marginals are identically distributed,
because this is true for $Q^N$. Moreover, by \cref{thm:mainGrottoRomito} the law at each fixed $t$ is the Energy-Enstrophy measure $(\omega_t)_\# Q=\mu_\beta$.

By Skorohod's representation theorem, there exists a new probability space $(\hat{\Omega},\hat{\mathcal{F}},\hat{\PP})$ and random variables $\hat{\omega}^{N_k},\ \hat{\omega}\in \XX$ such that $\LL(\hat{\omega}^{N_k})=Q^{N_k},\ \LL(\hat{\omega})=Q.$ Moreover $\hat{\omega}^{N_k}\rightarrow\hat{\omega}$ in $\XX$ $\hat{P}$-a.s. We already know that $\hat{\omega}$ has all marginals distributed according to the Energy-Enstrophy measure: we are left to show that it satisfies the weak formulation \eqref{eq:weakvortstoch} in \cref{def:energyenstrophysol}.
Denote by $({\Omega}^\prime,{\mathcal{F}}^\prime,{\PP}^\prime)$ a probability space where for each $N$ it is defined a random permutation ${s}^\prime_N:{\Omega}^\prime\rightarrow S_N$ uniformly distributed. Define the new probability space \begin{align*}
    (\Omega,\mathcal{F},\PP):=(\hat{\Omega}\times {\Omega}^\prime,\mathcal{F}\otimes{\mathcal{F}}^\prime,\hat{\PP}\otimes {\PP}^\prime)
\end{align*}
and the new processes \begin{align*}
    \omega^N=\hat{\omega}^N\circ \pi_1,\quad s_N={s}^\prime_N\circ \pi_2,
\end{align*}
where $\pi_1$ and $\pi_2$ are the projections on the first and second coordinate of $\hat{\Omega}\times {\Omega}^\prime$. 
Lastly, we denote by $\xxi_{s_N}$ the random vector
\begin{align*}
    \xxi_{s_N}=\left(\xi_{s_N(1)},\dots, \xi_{s_N(N)}\right).
\end{align*} Notice that the properties of convergence and of the laws of the processes $\omega^{N_k}$ and $\omega$ are the same as those of $\hat{\omega}^{N_k}$ and $\hat{\omega}$. In order to conclude the proof of \cref{thm:limit}, we first characterize $\omega^N_k$. For this reason we introduce the exchangeable measure on $\R^N\times D^N$ equal to   
\begin{align*}
    d\nu_{\E,\beta}^N(\bm\zeta,\x_N)=\frac{1}{Z_{\beta}^N}\sum_{\sigma\in S_N}\frac{1_{\{\xxi_{\sigma}=\bm\zeta\}}}{N!} e^{-\frac{\beta}{N} H(\xxi_{\sigma},\x_N)}d\x_N,
\end{align*}
where \begin{align*}
H(\xxi_{\sigma},\x_N)=\sum_{i<j}\xi_{\sigma(i)}\xi_{\sigma(j)}G(x_i,x_j)+\frac{1}{2}\sum_{i=1}^N \xi_{\sigma(i)}^2 g(x_i,x_i).
\end{align*}
With this notation in mind, we can adapt the proof of \cite[Lemma 28]{Fla2018} easily, obtaining
\begin{proposition}\label{characterization approximating processes}
    The process $\omega^{N_k}$ on the new probability space can be represented in the form \begin{align*}
        \omega^{N_k}=\frac{1}{\sqrt{N}}\sum_{i=1}^{N_k}\Xi_i \delta_{x^{i,N_k}_t}
    \end{align*}
    where 
    \begin{align*}
        \left((\Xi_{1},x^{1,N_k}_0),\dots, (\Xi_{N_k},x^{N_k,N_k}_0)\right)
    \end{align*} 
    is a random vector with law $\nu_{\E,\beta}^{N_k}$, where $\Xi=\left(\Xi_1,\dots, \Xi_{N_k}\right)$ and \\ $(x^{1,N_k}_\cdot,\dots, x^{N_k,N_k}_\cdot)$ solves system \eqref{eq:pv} with initial condition \\ $(x^{1,N_k}_0,\dots, x^{N_k,N_k}_0)$. In particular fixing a canonical ordering, e.g. conditioning all the events to the fact that $\Xi_1=\xi_1,\dots,\Xi_{N_k}=\xi_{N_k}$, then
     \begin{align*}
        &\left(x^{1,N_k}_0,\dots, x^{N_k,N_k}_0\right)
    \end{align*}
 has law $\nu_{\beta}^{N_k}$ and $(x^{1,N_k}_\cdot,\dots, x^{N_k,N_k}_\cdot)$ solves system \eqref{eq:pv} with initial condition $(x^{1,N_k}_0,\dots, x^{N_k,N_k}_0)$.
\end{proposition}
In order to conclude the proof of \cref{thm:limit} we are left to show that for every $\phi\in \mathcal{D}$, $\PP$-a.s. \eqref{eq:weakvorticityformulation} holds uniformly in time. We recall that \begin{align}\label{convergence in which norm}
    \omega^{N_k}\rightarrow \omega \quad \text{in }\XX\quad \PP-a.s.
\end{align}Actually we will show that given $\phi\in \mathcal{D}$ it holds 
\begin{align*}
    \EE\left[\abs{  I^{1}_{\omega_t}(\phi)- I^{1}_{\omega_0}(\phi)-\int_0^t I^2_{\omega_s}(H_\phi+h_\phi)\ ds}\wedge 1\right]=0.
\end{align*}
This implies that 
\begin{align*}
     I^{1}_{\omega_t}(\phi)- I^{1}_{\omega_0}(\phi)-\int_0^t I^2_{\omega_s}(H_\phi+h_\phi)\ ds=0 \quad \PP-a.s.
\end{align*}
Then the required uniformity in time follows from the continuity in all terms appearing in \eqref{eq:weakvorticityformulation}.
We start observing that, from the characterization of $\omega^{N_k}$ guaranteed by \cref{characterization approximating processes}, \cref{Well-posedness point vortices} and the discussion in \cref{subsec:weak_vorticity_formulation}, then $\omega^{N_k}$ satisfies  \begin{align*}
    \langle \omega^{N_k}_t,\phi\rangle-\langle \omega^{N_k}_0,\phi\rangle-\int_0^t \langle \omega^{N_k}_s\otimes \omega^{N_k}_s,H_\phi+h_\phi\rangle ds=0 \quad\PP-a.s.
\end{align*}
Moreover, thanks to elementary inequality $\lvert x+y\rvert\wedge 1\leq \lvert x\rvert\wedge 1+\lvert y\rvert \wedge 1$,
it holds 
\begin{align*}
    &\EE\left[\abs{ I^{1}_{\omega_t}(\phi)- I^{1}_{\omega_0}(\phi)- \int_0^t I^2_{\omega_s}(H_\phi+h_\phi)\ ds }\wedge 1\right]\\ & \leq \EE\left[\abs{ I^1_{\omega_t}(\phi)-\langle \omega^{N_k}_t,\phi\rangle}\wedge 1\right]+\EE\left[\abs{ I^1_{\omega_0}(\phi)-\langle \omega^{N_k}_0,\phi\rangle}\wedge 1\right]\\ & + \EE\left[\abs{ \int_0^t I^2_{\omega_s}(H_\phi+h_\phi)-\langle \omega^{N_k}_s\otimes \omega^{N_k}_s,H_\phi+h_\phi\rangle ds}\wedge 1\right]\\ & = I^{N_k}_1+I^{N_k}_2+I^{N_k}_3.
\end{align*}
$I_1^{N_k}+I_2^{N_k}\rightarrow 0$ easily thanks to \eqref{convergence in which norm} and Lebesgue's dominated convergence Theorem. 
For what concerns $I^{N_k}_3$, for each $\delta>0$ we introduce the smooth approximation of $H_\phi^{\delta}$ and $h_\phi^{\delta}$ defined in \cref{subse:ApproxIntFun}. 
Now either $H_\phi^{\delta}$ and $H_\phi^{\delta}$ and their derivatives of each order are smooth and equal to $0$ on $\triangle\cup \left(\partial D\times D\right)\cup \left( D\times \partial D\right).$ Therefore, thanks to \eqref{eq:equivalence_stoch_integral},
\begin{align*}
    &\abs{ \int_0^t\omega^{N_k}_s\otimes \omega^{N_k}_s,H^{\delta}_{\phi}\rangle-I^2_{\omega_s}(H^{\delta}_{\phi}) ds }
    =\abs{ \int_0^t\omega^{N_k}_s\otimes \omega^{N_k}_s-\omega_s\otimes \omega_s,H^{\delta}_{\phi}\rangle ds }\\
    &\qquad \leq \abs{ \int_0^t\omega^{N_k}_s\otimes (\omega^{N_k}_s-\omega_s),H^{\delta}_{\phi}\rangle ds } +
    \abs{ \int_0^t(\omega^{N_k}_s-\omega_s)\otimes \omega_s,H^{\delta}_{\phi}\rangle ds }\\ 
    &\qquad \leq t\norm{ \omega^{N_k}}_{C([0,T];H^{-1-\delta})}\norm{ \omega^{N_k}-\omega}_{C([0,T];H^{-1-\delta})}
    \norm{ H^{\delta}_{\phi}}_{H^{1+\delta}_0(D;H^{1+\delta}_0)}\\ 
    &\qquad \quad +t\norm{ \omega}_{C([0,T];H^{-1-\delta})}\norm{ \omega^{N_k}-\omega}_{C([0,T];H^{-1-\delta})}\norm{ H^{\delta}_{\phi}}_{H^{1+\delta}_0(D;H^{1+\delta}_0)}\\ 
    &\qquad \leq T \norm{ H_\phi^{\delta}}_{C^4(\bar{D}^2)}\norm{ \omega^{N_k}-\omega}_{C([0,T];H^{-1-\delta})}\\
    &\qquad \quad \times
    \left(\norm{ \omega}_{C([0,T];H^{-1-\delta})}+\norm{ \omega^{N_k}}_{C([0,T];H^{-1-\delta})}\right),
\end{align*}
the latter converging to $0$ almost surely; the convergence of the term involving $h_\phi^{\delta}$ is analogous. 
We thus deduce by dominated convergence that:
\begin{equation*}
    \forall \delta>0, \quad
    \EE\left[\abs{\int_0^t I^2_{\omega_s}(H_\phi^{\delta}+h_\phi^{\delta})-\langle \omega^{N_k}_s\otimes \omega^{N_k}_s,H_\phi^{\delta}+h_\phi^{\delta}\rangle ds}\wedge 1\right]\rightarrow 0,
\end{equation*}
hence
\begin{multline*}
    \limsup_{k\rightarrow +\infty} I_3^{N_k}\leq \EE\left[\abs{\int_0^t I^2_{\omega_s}(H_\phi+h_\phi-H_\phi^{\delta}+h_\phi^{\delta}) ds}\wedge 1\right] \\  +\limsup_{k\rightarrow+\infty}\EE\left[\abs{ \int_0^t\langle \omega^{N_k}_s\otimes \omega^{N_k}_s,H_\phi+h_\phi-H_\phi^{\delta}+h_\phi^{\delta}\rangle ds}\wedge 1\right]  = I_{3,1}+I_{3,2}.
\end{multline*}
The term $I_{3,2}$ can be made arbitrarily small choosing $\delta$ small enough, thanks to the characterization of $\omega^{N_k}$ guaranteed by \cref{characterization approximating processes}, \cref{thm:timebounds} and the convergence properties of $H^{\delta}_{\phi}$ (resp. $h^{\delta}_{\phi}$) to $H_\phi$ (resp. $h_\phi$) described in \cref{subsec:weak_vorticity_formulation}. Indeed, setting 
\begin{align*}
\HH^{\delta,\phi}=H_\phi+h_\phi-H_\phi^{\delta}+h_\phi^{\delta},
\end{align*} 
it holds, for each $\epsilon<1/2$ fixed,
\begin{align*}
    I_{3,2}&\leq \limsup_{k\rightarrow+\infty}\int_0^t \EE\left[\abs{ \langle\omega^{N_k}_s\otimes \omega^{N_k}_s,\HH^{\delta,\phi}\rangle }\right] ds\\ & \leq \limsup_{k\rightarrow+\infty}\int_0^t \EE\left[\abs{ \langle\omega^{N_k}_s\otimes \omega^{N_k}_s,\HH^{\delta,\phi}\rangle }^2\right]^{1/2} ds\\ & =\limsup_{k\rightarrow +\infty} \int_0^t \expt{\expt{\abs{ \langle\omega^{N_k}_s\otimes \omega^{N_k}_s,\HH^{\delta,\phi}\rangle }^2\Bigr| \Xi_1,\dots,\Xi_{N_k}}}^{1/2} ds,
\end{align*}
from which
\begin{multline*}
    I_{3,2} \lesssim_{\beta,\epsilon,D,T} \limsup_{k\rightarrow+\infty}\bigg(\norm{ \HH^{\delta,\phi}}_{L^{2/\epsilon}}+\frac{\operatorname{sup}_{x\in \bar{D}}\lvert \HH^{\delta,\phi}(x,x)\rvert}{N_k^{1/2}}\\ +\frac{\operatorname{sup}_{x\in\bar{D}}\lvert \HH^{\delta,\phi}(x,x)\rvert^{1/2}\norm{ \HH^{\delta,\phi}}_{L^{1/\epsilon}}^{1/2}}{N_k^{1/2-\epsilon}}\bigg) = \norm{ \HH^{\delta,\phi}}_{L^{2/\epsilon}}.
\end{multline*}
The analysis of $I_{3,1}$ is analogous, and proceeds by exploiting relation \eqref{eq:ItoIsometry} instead of \cref{thm:timebounds}.

\begin{acknowledgements}
    F. G. wishes to thank Alessandro Iraci for insightful remarks on the combinatorial arguments in \Cref{sec:fluctuationbounds}.
    M. R. acknowledges the partial support of PNRR - M4C2 - Investimento 1.3, Partenariato Esteso PE00000013 - "FAIR - Future Artificial Intelligence Research" - Spoke 1 "Human-centered AI", funded by the European Commission under the NextGeneration EU programme, and of the University of Pisa, through project PRA\_2022\_85.
\end{acknowledgements}

\bibliographystyle{plain}

\end{document}